\documentclass[a4paper,11pt]{amsart}

\usepackage[french, english]{babel} 
\usepackage[T1]{fontenc}
\usepackage[ansinew]{inputenc}

\date{}
\title[Stabilization and approximate null-controllability for diffusive equations]{Stabilization and approximate null-controllability for a large class of diffusive equations from thick control supports}

\author{Paul Alphonse}
\address{(Paul Alphonse) Université de Lyon, ENSL, UMPA - UMR 5669, F-69364 Lyon}
\email{paul.alphonse@ens-lyon.fr}

\author{Jérémy Martin}
\address{(Jérémy Martin) Univ Rennes, CNRS, IRMAR - UMR 6625, F-35000 Rennes}
\email{jeremy.martin@ens-rennes.fr}


\makeatletter
	\@namedef{subjclassname@2020}{\textup{2020} Mathematics Subject Classification}
\makeatother


\usepackage[top=3cm, bottom=2cm, left=3cm, right=3cm]{geometry}	

\usepackage{enumitem}

\usepackage{array}											

\usepackage{amsfonts}
\usepackage{amsmath}
\usepackage{amsthm}
\usepackage{amssymb}
\usepackage{mathtools}

\usepackage{multicol}

\usepackage{bbm}

\usepackage{stmaryrd}

\usepackage[scr]{rsfso}

\usepackage{comment}

\usepackage[dvipsnames]{xcolor}

\usepackage{hyperref}

\hypersetup{	
colorlinks=true,
breaklinks=true,
urlcolor= red,
linkcolor= red,
citecolor=ForestGreen
}

\numberwithin{equation}{section}

\newtheorem{thm}{Theorem}[section]
\newtheorem{prop}[thm]{Proposition}

\newtheorem{lem}[thm]{Lemma}
\newtheorem{cor}[thm]{Corollary}
\theoremstyle{definition}
\newtheorem{dfn}[thm]{Definition}
\newtheorem{ex}[thm]{Example}
\newtheorem{rk}[thm]{Remark}

\DeclareMathOperator{\Leb}{Leb}
\DeclareMathOperator{\Supp}{Supp}
\DeclareMathOperator{\Reelle}{Re}

\begin{document}

\sloppy

\selectlanguage{english}

\begin{abstract} We prove that the thickness property is a necessary and sufficient geometric condition that ensures the (rapid) stabilization or the approximate null-controllability with uniform cost of a large class of evolution equations posed on the whole space $\mathbb R^n$. These equations are associated with operators of the form $F(\vert D_x\vert)$, the function $F:[0,+\infty)\rightarrow\mathbb R$ being continuous and bounded from below. We also provide explicit feedbacks and constants associated with these stabilization properties. The notion of thickness is known to be a necessary and sufficient condition for the exact null-controllability of the fractional heat equations associated with the functions $F(t) = t^{2s}$ in the case $s>1/2$. Our results apply in particular for this class of equations, but also for the half heat equation associated with the function $F(t) = t$, which is the most diffusive fractional heat equation for which exact null-controllability is known to fail from general thick control supports. 
\end{abstract}

\maketitle

\section{Introduction}

This paper is devoted to investigate the stabilization and approximate null-controllability for control systems of the following form:
\begin{equation}\label{01062020E1}\tag{$E_F$}
\left\{\begin{aligned}
	& \partial_tf(t,x) + F(\vert D_x\vert) f(t,x) = h(t,x)\mathbbm 1_{\omega}(x),\quad t>0,\ x\in\mathbb R^n, \\
	& f(0,\cdot) = f_0\in L^2(\mathbb R^n),
\end{aligned}\right.
\end{equation}
where the operator $F(\vert D_x\vert)$ is the Fourier multiplier associated with the symbol $F(\vert\xi\vert)$, with $\vert\cdot\vert$ the canonical Euclidean norm in $\mathbb R^n$, the function $F:[0,+\infty)\rightarrow\mathbb R$ being continuous and bounded from below, and $\omega\subset\mathbb R^n$ is a Borel set with positive Lebesgue measure. 

The study of the (rapid) stabilization and the (approximate) null-controllability of evolution equations of the form \eqref{01062020E1} has been much addressed recently \cite{A, AB, HWW, K, L, LWXY, MR2679651, MR4041279}. The Schr\"odinger counterparts of these equations and the same equations posed on bounded domains have also been studied, respectively in \cite{MPS2} and \cite{MR1312710, Xiang}. In this work, we consider control supports $\omega\subset\mathbb R^n$ which are thick:

\begin{dfn} Given $\gamma\in(0,1)$ and $L>0$, the set $\omega\subset\mathbb R^n$ is said to be $\gamma$-\textit{thick} at scale $L$ when it is measurable and satisfies
$$\forall x\in\mathbb R^n,\quad\Leb(\omega\cap(x + [0,L]^n))\geq\gamma L^n,$$
where $\Leb$ denotes the Lebesgue measure in $\mathbb R^n$.
\end{dfn}

This notion of thickness has appeared to play a key role in the null-controllability theory since the works \cite{MR3816981, WZ}, where the authors established that it is a necessary and sufficient geometric condition that ensures the null-controllability of the heat equation posed on $\mathbb R^n$, which is the equation \eqref{01062020E1} associated with the function $F(t) = t^2$. The same phenomenon holds true more generally for the evolution equations associated with fractional Laplacians $(-\Delta)^s$ (case where $F(t) = t^{2s}$) in the same setting and when $s>1/2$, as proven in \cite{AB}, and also quite surprisingly for the Schr\"odinger counterparts of these equations in the one dimensional setting and when $s\geq1/2$, see \cite{MPS2} (Corollary 2.8). It is also known from the works \cite{Ko, K} that in the case $0<s\le1/2$, the fractional heat equation (\hyperref[01062020E1]{$E_{t^{2s}}$}) is not anymore null-controllable from thick control supports in general. Other classes of degenerate parabolic equations of hypoelliptic type, as evolution equations associated with accretive quadratic operators or (non-autonomous) Ornstein-Uhlenbeck operators, were also proven to be null-controllable from thick control supports, see e.g. \cite{A, BEPS, BJKPS}. In this work, we prove that a very general class of equations of the form \eqref{01062020E1} is approximately null-controllable with uniform cost from the control support $\omega\subset\mathbb R^n$ if and only if $\omega$ is a thick set. Our results hold in particular for the half heat equation (\hyperref[01062020E1]{$E_t$}) which is not null-controllable from general thick control supports (at least when $n=1$, the case $n\geq2$ remaining open for the moment), see e.g. \cite{K} (Theorem 2.3) or \cite{L} (Theorem 1.1).

The study of the (rapid) stabilization of the control system \eqref{01062020E1}, as for it, has been addressed very recently in the works \cite{HWW, LWXY, MR4041279}. It has been proven in \cite{HWW} (Theorem 1.1) that for all $s>0$, the fractional heat equation (\hyperref[01062020E1]{$E_{t^{2s}}$}) is exponentially stabilizable from the support $\omega$ if and only if $\omega$ is thick. It is also known from \cite{LWXY} (Example 1) that the very same equation (\hyperref[01062020E1]{$E_{t^{2s}}$}) is rapidly stabilizable from complements of Euclidean balls in $\mathbb R^n$ when $0<s<1$. In this paper, we establish that the control system \eqref{01062020E1} is exponentially stabilizable from $\omega$ if and only if $\omega$ is a thick set when $\inf F \le 0$ (in the case $\inf F>0$, the control system \eqref{01062020E1} is stable) and $\liminf_{+\infty} F>-\inf F$. Moreover, we provide explicit formulas for the feedbacks $K$ and the constants associated with this stabilization, which allows us to prove that when $\lim_{+\infty} F = +\infty$, the control system \eqref{01062020E1} is rapidly stabilizable from $\omega$ if and only if $\omega$ is thick. In particular, we recover \cite{HWW} (Theorem 1.1) (with new explicit feedbacks) and we generalize \cite{LWXY} (Example 1). We also prove that when $\omega$ is not dense in $\mathbb R^n$, the equation \eqref{01062020E1} is never rapidly stabilizable in the particular case where $F$ admits a finite limit at $+\infty$.

In a nutshell, our results highlight the importance of the notion of thickness not only in the null-controllability theory, but also for properties of stabilization and approximate null-controllability with uniform cost, as it turns out to be a necessary and sufficient geometric condition ensuring these two properties for a large class of diffusive equations \eqref{01062020E1}.

\subsubsection*{Outline of the work} In Section \ref{results}, we present in details the main results contained in this work. Section \ref{stab} is devoted to the proofs of the results concerning the stabilization and the rapid stabilization of the control system \eqref{01062020E1}. Basic properties of quasi-analytic sequences are presented in Section \ref{QA}, which allow to establish the results concerning the approximate null-controllability of the evolution equation \eqref{01062020E1} in Section \ref{cunc}. Finally, Section \ref{appendix} is an Appendix concerning the proof of an observability result used in Section \ref{cunc}.

\subsubsection*{Notations} The following notations and conventions will be used all over this work:
\begin{enumerate}[label=\textbf{\arabic*.},leftmargin=* ,parsep=2pt,itemsep=0pt,topsep=2pt]
	\item The canonical Euclidean scalar product of $\mathbb R^n$ is denoted by $\cdot$ and $\vert\cdot\vert$ stands for the associated canonical Euclidean norm.
	\item For all measurable subsets $\omega\subset\mathbb R^n$, the inner product of $L^2(\omega)$ is defined by
$$\langle u,v\rangle_{L^2(\omega)} = \int_{\omega}u(x)\overline{v(x)}\ \mathrm dx,\quad u,v\in L^2(\omega),$$
while $\Vert\cdot\Vert_{L^2(\omega)}$ stands for the associated norm. Moreover, $\mathcal L(L^2(\omega))$ stands for the set of bounded operators on $L^2(\omega)$.
	\item For all function $u\in\mathscr S(\mathbb R^n)$, the Fourier transform of $u$ is denoted $\widehat u$ or $\mathscr Fu$, and defined by
$$\widehat u(\xi) = (\mathscr Fu)(\xi) = \int_{\mathbb R^n}e^{-ix\cdot\xi}u(x)\ \mathrm dx,\quad \xi\in\mathbb R^n.$$
With this convention, Plancherel's theorem states that 
$$\forall u\in L^2(\mathbb R^n),\quad \Vert\widehat u\Vert_{L^2(\mathbb R^n)} = (2\pi)^{n/2}\Vert u\Vert_{L^2(\mathbb R^n)}.$$
	\item We denote by $\nabla_x$ the gradient and we set $D_x = -i\nabla_x$. Moreover, $F(\vert D_x\vert)$ stands for the Fourier multiplier associated with the symbol $F(\vert\xi\vert)$ for all continuous function $F:[0,+\infty)\rightarrow\mathbb R$.
	\item For all measurable subsets $\omega\subset\mathbb R^n$, $\mathbbm1_{\omega}$ stands for the characteristic function of $\omega$.
\end{enumerate}


\section{Statement of the main results}\label{results}

This section is devoted to present in details the main results contained in this work. Let us begin by defining precisely the different concepts related to the control system \eqref{01062020E1} we are interested in:
\begin{enumerate}[label=$(\roman*)$,leftmargin=* ,parsep=2pt,itemsep=0pt,topsep=2pt]
	\item The control system \eqref{01062020E1} is said to be \textit{null-controllable} from the control support $\omega$ in time $T>0$ when for all $f_0\in L^2(\mathbb R^n)$, there exists a control $h\in L^2((0,T)\times\omega)$ such that the mild solution of \eqref{01062020E1} satisfies $f(T,\cdot) = 0$.
	\item The control system \eqref{01062020E1} is said to be \textit{approximately null-controllable} from the control support $\omega$ in time $T>0$ if for all $\varepsilon>0$ and $f_0\in L^2(\mathbb R^n)$, there exists a control $h\in L^2((0,T)\times\omega)$ such that the mild solution of \eqref{01062020E1} satisfies $\Vert f(T,\cdot)\Vert_{L^2(\mathbb R^n)}\le\varepsilon$.
	\item The control system \eqref{01062020E1} is said to be \textit{approximately null-controllable with uniform cost} from the control support $\omega$ in time $T>0$ if for all $\varepsilon>0$, there exists a positive constant $C_{\varepsilon,T}>0$ such that for all $f_0\in L^2(\mathbb R^n)$, there exists a control $h\in L^2((0,T)\times\omega)$ such that the mild solution of \eqref{01062020E1} satisfies 
	$$\Vert f(T,\cdot)\Vert_{L^2(\mathbb R^n)}\le\varepsilon\Vert f_0\Vert_{L^2(\mathbb R^n)},$$
	with moreover
	$$\int_0^T\Vert h(t,\cdot)\Vert^2_{L^2(\omega)}\ \mathrm dt\le C_{\varepsilon,T}\Vert f_0\Vert^2_{L^2(\mathbb R^n)}.$$
	\item The control system \eqref{01062020E1} is said to be \textit{exponentially stabilizable} from the control support $\omega$ at rate $\alpha>0$ if there exist a positive constant $M_{\alpha}\geq1$ and a feedback $K_{\alpha}\in\mathcal L(L^2(\mathbb R^n))$ such that for all $t\geq0$,
	\begin{equation}\label{12042021E3}
		\big\Vert e^{-t(F(\vert D_x\vert)+\mathbbm 1_{\omega}K_{\alpha})}\big\Vert_{\mathcal L(L^2(\mathbb R^n))}\le M_{\alpha}e^{-\alpha t}.
	\end{equation}
	When the feedback $K_{\alpha}$ can be chosen equal to zero, the control system \eqref{01062020E1} is said to be \textit{stable}. The existence of the semigroup generated by the operator $F(\vert D_x\vert)+\mathbbm 1_{\omega}K_{\alpha}$ is ensured by the theory of bounded perturbation of semigroups, see e.g. \cite{MR1721989} (Theorem III.1.3).
	\item The control system \eqref{01062020E1} is said to be \textit{rapidly stabilizable} from the control support $\omega$ if it is exponentially stabilizable from $\omega$ at any rate $\alpha>0$.
\end{enumerate}

\subsection{Stabilization} First of all, we are interested in tackling stabilization issues for the evolution system \eqref{01062020E1}. Let us begin by noticing that when $\inf F>0$, we get from Plancherel's theorem that
$$\forall t\geq0,\quad\big\Vert e^{-tF(\vert D_x\vert)}\big\Vert_{\mathcal L(L^2(\mathbb R^n))}\le e^{-(\inf F)t},$$
so the control system \eqref{01062020E1} is stable. The interesting case is therefore when $\inf F\le0$. In this case, we prove that the thickness of the support $\omega\subset\mathbb R^n$ is a necessary geometric condition that ensures the stabilization of the equation \eqref{01062020E1}, and a sufficient one when assuming in addition that $\liminf_{+\infty}F>\vert\inf F\vert$. We also provide explicit feedbacks and quantitative estimates associated with this stabilization.

\begin{thm}\label{11112020T1} Let $F : [0,+\infty)\rightarrow\mathbb R$ be a continuous function bounded from below  and $\omega\subset\mathbb R^n$ be a measurable set.
\begin{enumerate}[label=$(\roman*)$,leftmargin=* ,parsep=2pt,itemsep=0pt,topsep=2pt]
	\item If $\inf F\le 0$ and the evolution system \eqref{01062020E1} is exponentially stabilizable from $\omega$, then the set $\omega$ is thick.
	\item When $\liminf_{+\infty}F>\inf F$ and $\omega$ is a thick set, then there exist some positive constants $C = C(\omega)\geq1$ and $R_0 = R_0(F)>0$ such that for all $R\geq R_0$ and $t\geq0$,
\begin{equation}\label{29102020T1}
	\big\Vert e^{-t(F(\vert D_x\vert)+Ce^{CR}(\alpha_R-\inf F)\mathbbm 1_{\omega}K_R)}\big\Vert_{\mathcal L(L^2(\mathbb R^n))}\le \sqrt 2Ce^{CR}e^{-(\alpha_R+\inf F)t/2},
\end{equation}
where we set $\alpha_R = \inf_{r\geq R}F(r)$ and where $K_R$ stands for the following orthogonal projection
$$K_R : L^2(\mathbb R^n)\rightarrow\big\{f\in L^2(\mathbb R^n) : \Supp\widehat f\subset\overline{B(0,R)}\big\},$$
with $\overline{B(0,R)}$ the closed Euclidean ball centered in $0$ with radius $R>0$.
	\item When $\liminf_{+\infty}F>\vert\inf F\vert$ and $\omega$ is a thick set, then the evolution system \eqref{01062020E1} is exponentially stabilizable from $\omega$.
\end{enumerate}
\end{thm}

Let us check that the assertion $(iii)$ in Theorem \ref{11112020T1} is a straightforward consequence of the assertion $(ii)$ in the same result. Obviously, the assumption $\liminf_{+\infty}F>\vert\inf F\vert$ is read as $\liminf_{+\infty}F>\inf F$ and $\liminf_{+\infty}F>-\inf F$. On the one hand, the assumption $\liminf_{+\infty}F>\inf F$ implies that the estimates \eqref{29102020T1} hold according to the assertion $(ii)$. On the other hand, assuming that $\liminf_{+\infty}F>-\inf F$, we get that $\alpha_R+\inf F>0$ when $R\gg1$ is large enough. The estimates \eqref{29102020T1} then imply that the evolution equation \eqref{01062020E1} is exponentially stabilizable, according to the definition \eqref{12042021E3} presented in the beginning of this section.

By gathering the results stated in the assertions $(i)$ and $(iii)$ in Theorem \ref{11112020T1}, we directly obtain the following corollary:

\begin{cor}\label{08012021E1} Let $F : [0,+\infty)\rightarrow\mathbb R$ be a continuous function bounded from below  satisfying $\inf F\le 0$ and $\liminf_{+\infty}F>-\inf F$, and $\omega\subset\mathbb R^n$ be a measurable set. The evolution system \eqref{01062020E1} is exponentially stabilizable from $\omega$ if and only if $\omega$ is thick.
\end{cor}



It is a very interesting issue to know whether the control system \eqref{01062020E1} is exponentially stabilizable when $\inf F\le 0$ and $\liminf_{+\infty} F\le-\inf F$. We shall not tackle such a question in this work.

As a consequence of the quantitative stabilization estimates \eqref{29102020T1}, we directly obtain the following result concerning the rapid stabilization of the evolution system \eqref{01062020E1} under the assumption $\lim_{+\infty} F = +\infty$, by applying Theorem \ref{11112020T1} to the function $F-\inf F$.

\begin{cor}\label{complete_stabilization} Let $F : [0,+\infty)\rightarrow\mathbb R$ be a continuous function bounded from below satisfying $\lim_{+\infty} F = +\infty$, and $\omega\subset\mathbb R^n$ be a measurable set. The evolution system \eqref{01062020E1} is rapidly stabilizable from $\omega$ if and only if $\omega$ is thick.
\end{cor}

\begin{ex} For all positive real numbers $s>0$, let us consider the function $F_s:[0,+\infty)\rightarrow[0,+\infty)$ defined for all $t\geq0$ by $F_s(t) = t^{2s}$. We also consider $\omega\subset\mathbb R^n$ a measurable set with positive Lebesgue measure. It follows from Corollaries \ref{08012021E1} and \ref{complete_stabilization} that the associated control system (\hyperref[01062020E1]{$E_{F_s}$}) is exponentially stabilizable from the control support $\omega$ if and only if it is rapidly stabilizable from $\omega$ if and only if $\omega$ is a thick set. Moreover, we deduce from Theorem \ref{11112020T1} that when $\omega$ is thick, there exist a positive constant $C\geq1$ and $R_0>0$ such that for all $R\geq R_0$ and $t\geq0$,
$$\big\Vert e^{-t((-\Delta)^s+Ce^{CR}R^{2s}\mathbbm 1_{\omega}K_R)}\big\Vert_{\mathcal L(L^2(\mathbb R^n))}\le\sqrt 2 Ce^{CR}e^{-R^{2s}t/2},$$
where $K_R$ stands for the following orthogonal projection
$$K_R : L^2(\mathbb R^n)\rightarrow\big\{f\in L^2(\mathbb R^n) : \Supp\widehat f\subset\overline{B(0,R)}\big\}.$$
These explicit stabilization estimates allow to recover \cite{HWW} (Theorem 1.1) and also to generalize \cite{LWXY} (Example 1).
\end{ex}

In the case where $\lim_{+\infty} F<+\infty$, we only provide a necessary condition for the control system \eqref{01062020E1} to be rapidly stabilizable. The following result implies in particular that when the function $F$ has a finite limit in $+\infty$ and when the support $\omega\subset\mathbb R^n$ is not dense in $\mathbb R^n$, then the equation \eqref{01062020E1} is not rapidly stabilizable from $\omega$. 

\begin{prop}\label{negativ_result} Let $F : [0,+\infty)\rightarrow\mathbb R$ be a continuous function bounded from below and $\omega\subset\mathbb R^n$ be a measurable set which is not dense in $\mathbb R^n$. We assume that $\lim_{+\infty}F$ exists and is a non-negative real number (the function $F$ is therefore bounded). Then, if $\alpha>\sup F$, the equation \eqref{01062020E1} is not exponentially stabilizable from $\omega$ at rate $\alpha$.
\end{prop}

\subsection{Cost-uniform approximate null-controllability}\label{approximate_cont_result} In the second part of this work, we study the cost-uniform approximate null-controllability of the equations \eqref{01062020E1}. We will not address this question for general continuous functions $F$ bounded from below, but only for the ones generating a \textit{quasi-analytic sequence}. Let us precisely define this class of functions. Associated with the function $F$ is the following log-convex sequence $\mathcal M^F$ whose elements $M^F_k$, assumed to be positive real numbers, are defined by
\begin{equation}\label{16122020E1}
	0<M^F_k = \sup_{r\geq0} r^ke^{-F(r)}<+\infty,\quad k\geq0.
\end{equation}
We say that the sequence $\mathcal M^F$ is \textit{quasi-analytic} when, for all real numbers $a<b$, the associated Denjoy-Carleman class
$$\mathcal C_{\mathcal M^F}([a,b])= \big\{ f \in\mathcal C^{\infty}([a,b],\mathbb C) : \forall k\geq0,\forall x\in[a,b],\ \vert f^{(k)}(x)\vert \le M_k^F\big\},$$
is quasi-analytic, meaning that any function $f$ in this class satisfying
$$\exists x_0\in[a,b],\forall k\geq0,\quad f^{(k)}(x_0) = 0,$$
is identically equal to zero. We refer to Section \ref{QA} where the notion of quasi-analytic sequence is discussed, and where Denjoy-Carleman's theorem is presented, giving a useful characterization of such sequences. 

When the function $F$ generates a quasi-analytic sequence $\mathcal M^F$ of positive real numbers, the solutions of the homogeneous counterpart of the equation \eqref{01062020E1} belong to quasi-analytic classes of functions (see Subsection~\ref{approx_cont_proof} for more details). By taking advantage of this quasi-analytic regularity, we prove that the notion of thickness is a necessary and sufficient geometric condition that ensures the cost-uniform approximate null-controllability of the evolution equations \eqref{01062020E1} in any positive time.

\begin{thm}\label{21102020T1} Let $F : [0,+\infty) \rightarrow \mathbb R$ be a continuous function bounded from below and $\omega\subset\mathbb R^n$ be a measurable set. We assume that the sequence $\mathcal M^F$ associated with the function $F$ defined in \eqref{16122020E1} is a quasi-analytic sequence of positive real numbers. Then, for all positive time $T>0$, the diffusive equation \eqref{01062020E1} is cost-uniformly approximately null-controllable from the control support $\omega$ in time $T$ if and only if $\omega$ is thick.
\end{thm}

The necessary part of Theorem \ref{21102020T1} is a consequence of the fact that the cost-uniform approximate null-controllability implies rapid stabilization, see Proposition \ref{06012021P1}, and is therefore a consequence of Corollary \ref{complete_stabilization} (notice that the assumption on $F$ in Theorem \ref{21102020T1} implies in particular that $\lim_{+\infty}F = +\infty$.)

Let us now present explicit examples of functions $F$ generating quasi-analytic sequences $\mathcal M^F$ and for which Theorem \ref{21102020T1} therefore applies.

\begin{ex}\label{06112020D1} Let us assume that the non-negative continuous function $F:[0,+\infty)\rightarrow[0,+\infty)$ satisfies $\Theta\le F$, where the weight $\Theta:[0,+\infty)\rightarrow[0,+\infty)$ verifies the following properties:
\begin{enumerate}[label=$(\roman*)$,leftmargin=* ,parsep=2pt,itemsep=0pt,topsep=2pt]
	\item $\Theta(0) = 0$ and $\Theta$ is non-decreasing with $\lim_{+\infty}\Theta = +\infty$,
	\item $\Theta$ is lower-semicontinuous and $t\in\mathbb R\mapsto\Theta(e^t)$ is convex,
	\item $$\int_0^{+\infty}\frac{\Theta(t)}{1+t^2}\ \mathrm dt = +\infty.$$
\end{enumerate}
It follows from the work \cite{JM} (see also Proposition \ref{06112020P1} in the present work) that the sequence $\mathcal M^{\Theta}$ associated with the weight $\Theta$, defined in \eqref{16122020E1}, is quasi-analytic. Moreover, we have $\mathcal M^F\le\mathcal M^{\Theta}$, since $\Theta\le F$, and Lemma \ref{17122020L1} implies that the sequence $\mathcal M^F$ is also quasi-analytic. We deduce that for all Borel set $\omega\subset\mathbb R^n$ and all positive time $T>0$, the equation \eqref{01062020E1} is cost-uniformly approximately null-controllable from the set $\omega$ in time $T$ if and only if $\omega$ is thick. A relevant particular example is when $F(t) = \Theta(t) = t$. Indeed, the associated evolution equation is the half heat equation (\hyperref[01062020E1]{$E_t$}) posed on the whole space associated with the operator $\sqrt{-\Delta}$, known to be not null-controllable from any non dense control support $\omega$ (at least when $n=1$), see \cite{K} (Theorem 2.3) or \cite{L} (Theorem 1.1). This evolution equation is then a relevant example where the thick condition fails to be sufficient for the (strong) null-controllability but appears to be necessary and sufficient for the cost-uniform approximate null-controllability.
\end{ex}

Actually, we are able to derive cost-uniform approximate null-controllability results for much less diffusive equations than the half heat equation (\hyperref[01062020E1]{$E_t$}), as illustrated in the two following examples.

\begin{ex}\label{21122020E1} Let $s\geq1$, $0 \leq \delta \leq 1$ be non-negative real numbers and $F_{s, \delta} : [0,+\infty)\rightarrow[0,+\infty)$ be the non-negative continuous function defined for all $t\geq0$ by
$$F_{s,\delta}(t) = \frac{t^s}{\log^{\delta}(e+t)}.$$
We check in Corollary \ref{21122020C1} that the associated sequence $\mathcal M^{F_{s,\delta}}$ defined in \eqref{16122020E1} is a quasi-analytic sequence of positive real numbers. Therefore, for all Borel set $\omega\subset\mathbb R^n$ and all positive time $T>0$, the equation (\hyperref[01062020E1]{$E_{F_{s,\delta}}$}) is cost-uniformly approximately null-controllable from the set $\omega$ in time $T$ if and only if $\omega$ is thick. 
\end{ex}

\begin{ex}\label{ex_diffusion}  Let $p\geq1$ be a positive integer and $F_p : [0,+\infty)\rightarrow[0,+\infty)$ be the non-negative continuous function defined for all $t\geq0$ by
\begin{equation*}
	F_p(t) = \frac t{g(t)(g\circ g)(t)... g^{\circ p}(t)},\quad\text{where}\quad g(t) = \log(e+t),
\end{equation*}
with $g^{\circ p} = g\circ\ldots\circ g$ ($p$ compositions). We check in Proposition \ref{19122020P1} that the associated sequence $\mathcal M^{F_p}$ defined in \eqref{16122020E1} is quasi-analytic. As a consequence, for all Borel set $\omega\subset\mathbb R^n$ and all positive time $T>0$, the equation (\hyperref[01062020E1]{$E_{F_p}$}) is cost-uniformly approximately null-controllable from the set $\omega$ in time $T$ if and only if $\omega$ is thick. 
\end{ex}
Regarding the weaker notion of approximate null-controllability presented at the beginning of this section, the geometry of the allowed control support is much simpler.  Indeed, the following proposition ensures that the control system \eqref{01062020E1} is approximately null-controllable in any positive time $T>0$ and from any measurable set $\omega\subset\mathbb R^n$ with positive Lebesgue measure when $F$ generates a log-convex quasi-analytic sequence $\mathcal M^F$.

\begin{prop}\label{approximate_controllability_weak} Let $F : [0,+\infty) \rightarrow \mathbb R$ be a continuous function bounded from below and $\omega\subset\mathbb R^n$ be a measurable set with positive Lebesgue measure. If the sequence $\mathcal M^F$ associated with the function $F$, defined in \eqref{16122020E1}, is a quasi-analytic sequence of positive real numbers, then for all positive time $T>0$, the diffusive equation \eqref{01062020E1} is approximately null-controllable from the support control $\omega$ in time $T$.
\end{prop}
In particular, diffusive equations discussed in Examples \ref{06112020D1}, \ref{21122020E1} and \ref{ex_diffusion} are approximately null-controllable in any positive time $T>0$ from any measurable subset $\omega\subset\mathbb R^n$ satisfying $\Leb(\omega)>0$. 

\section{(Rapid) Stabilization of diffusive equations}\label{stab}

The aim of this section is to prove Theorem \ref{11112020T1} and Proposition \ref{negativ_result} concerning the stabilization and the rapid stabilization properties of the following general control system
\begin{equation}\label{02062020E1}\tag{$E_F$}
\left\{\begin{aligned}
	& \partial_tf(t,x) + F(\vert D_x\vert) f(t,x) = h(t,x)\mathbbm 1_{\omega}(x),\quad t>0,\ x\in\mathbb R^n, \\
	& f(0,\cdot) = f_0\in L^2(\mathbb R^n),
\end{aligned}\right.
\end{equation}
where $F : [0,+\infty) \rightarrow \mathbb R$ is a continuous function bounded from below and $\omega\subset\mathbb R^n$ is a measurable set.

\subsection{Proof of Theorem \ref{11112020T1}: assertion $(i)$}\label{necessary_part} First of all, let us assume that the equation \eqref{02062020E1} is exponentially stabilizable from the set $\omega$, with the additional assumption that $\inf F\le0$. We aim at proving that the control support $\omega$ is then thick. To that end, we will use the following nice characterizations of exponential stabilization in terms of observability estimates, taken from the work \cite{MR4041279}:

\begin{thm}[Theorem 1 in \cite{MR4041279}]\label{29102020T2} The following assertions are equivalent:
\begin{enumerate}[label=$(\roman*)$,leftmargin=* ,parsep=2pt,itemsep=0pt,topsep=2pt]
	\item The evolution system \eqref{01062020E1} is exponentially stabilizable from $\omega$.
	\item For all $\varepsilon\in(0,1)$, there exist $T>0$  and $C>0$ such that for all $g\in L^2(\mathbb R^n)$,
	$$\big\Vert e^{-TF(\vert D_x\vert)}g\big\Vert^2_{L^2(\mathbb R^n)}\le C\int_0^T\big\Vert e^{-tF(\vert D_x\vert)}g\big\Vert^2_{L^2(\omega)}\ \mathrm dt 
	+ \varepsilon\Vert g\Vert^2_{L^2(\mathbb R^n)}.$$
	\item There exist $\varepsilon\in(0,1)$, $T>0$ and $C>0$ such that for all $g\in L^2(\mathbb R^n)$,
	$$\big\Vert e^{-TF(\vert D_x\vert)}g\big\Vert^2_{L^2(\mathbb R^n)}\le C\int_0^T\big\Vert e^{-tF(\vert D_x\vert)}g\big\Vert^2_{L^2(\omega)}\ \mathrm dt 
	+ \varepsilon\Vert g\Vert^2_{L^2(\mathbb R^n)}.$$
\end{enumerate}
\end{thm}

According to the above theorem, assuming that the equation \eqref{02062020E1} is exponentially stabilizable from $\omega$ is equivalent to assuming that there exist $\varepsilon\in(0,1)$, $T>0$ and $C>0$ such that for all $g\in L^2(\mathbb R^n)$,
\begin{equation}\label{19122019E4}
	\big\Vert e^{-TF(\vert D_x\vert)}g\big\Vert^2_{L^2(\mathbb R^n)}\le C\int_0^T\big\Vert e^{-tF(\vert D_x\vert)}g\big\Vert^2_{L^2(\omega)}\ \mathrm dt 
	+ \varepsilon\Vert g\Vert^2_{L^2(\mathbb R^n)}.
\end{equation}
The strategy consists in applying this observability estimate for well-chosen functions $g\in L^2(\mathbb R^n)$. This approach has especially been used in the works \cite{A, AB, BEPS, HWW}, in which exponential stabilization or null-controllability issues are studied for fractional heat equations or evolution equations associated with (non)-autonomous Ornstein-Uhlenbeck operators posed on the whole space $\mathbb R^n$.

Fixing $x_0\in\mathbb R^n$ and considering $\xi_0\in\mathbb R^n$ together with $l\gg1$ whose values will be adjusted later, we consider the Gaussian function $g_{l,\xi_0}$ defined by
$$\forall x\in\mathbb R^n,\quad g_{l,\xi_0}(x) = \frac1{l^n}\exp\bigg(ix\cdot\xi_0-\frac{\vert x-x_0\vert^2}{2l^2}\bigg).$$
Classical results concerning Fourier transform of Gaussian functions show that
\begin{equation}\label{06112020E2}
	\forall\xi\in\mathbb R^n,\quad\widehat g_{l,\xi_0}(\xi) = (2\pi)^{n/2}\exp\bigg(-ix_0\cdot(\xi-\xi_0)-\frac{l^2\vert\xi-\xi_0\vert^2}2\bigg).
\end{equation}
On the one hand, it follows from Plancherel's theorem that the left-hand side of the inequality \eqref{19122019E4} applied to the functions $g_{l,\xi_0}$ is a positive constant independent of the point $x_0$, denoted $\delta_{l,\xi_0}>0$ in the following and given by
\begin{align}\label{19122019E3}
	\delta_{l,\xi_0} = \big\Vert e^{-TF(\vert D_x\vert)}g_{l,\xi_0}\big\Vert^2_{L^2(\mathbb R^n)} 
	& = \int_{\mathbb R^n}\big\vert e^{-ix_0\cdot(\xi-\xi_0)}e^{-TF(\vert\xi\vert)}e^{-l^2\vert\xi-\xi_0\vert^2/2}\big\vert^2\ \mathrm d\xi \\
	& = \frac1{l^n}\int_{\mathbb R^n}\big\vert e^{-TF(\vert\xi/l + \xi_0\vert)}e^{-\vert\xi\vert^2/2}\big\vert^2\ \mathrm d\xi>0. \nonumber 
\end{align}
On the other hand, we get that the $L^2$-norm of the function $g_{l,\xi_0}$ also does not depend on the point $x_0\in\mathbb R^n$ and is given by the following Gaussian integral
\begin{equation}\label{27012021E1}
	\Vert g_{l,\xi_0}\Vert^2_{L^2(\mathbb R^n)} = \frac1{l^{2n}}\int_{\mathbb R^n}e^{-\vert x\vert^2/l^2}\ \mathrm dx = \bigg(\frac{\pi}{l^2}\bigg)^{n/2}.
\end{equation}
Let us check that the point $\xi_0\in\mathbb R^n$ and the large positive parameter $l\gg1$ can be adjusted so that $\delta_{l,\xi_0} - \varepsilon\Vert g_{l,\xi_0}\Vert^2_{L^2(\mathbb R^n)}>0$, that is, by \eqref{19122019E3} and \eqref{27012021E1},
\begin{equation}\label{06102020E1}
	\int_{\mathbb R^n}\big\vert e^{-TF(\vert\xi/l+\xi_0\vert)}e^{-\vert\xi\vert^2/2}\big\vert^2\ \mathrm d\xi>\varepsilon\pi^{n/2}.
\end{equation}
Since $\varepsilon\in(0,1)$ and the function $F$ satisfies $\inf F\le0$, we can assume that the point $\xi_0\in\mathbb R^n$ is chosen in order to satisfy $e^{-2TF(\vert\xi_0\vert)}>\varepsilon$. Since the function $F$ is bounded from below, the dominated convergence theorem then implies that
\begin{align*}
	\lim_{l\rightarrow+\infty}\int_{\mathbb R^n}\big\vert e^{-TF(\vert\xi/l+\xi_0\vert)}e^{-\vert\xi\vert^2/2}\big\vert^2\ \mathrm d\xi 
	& = e^{-2TF(\vert\xi_0\vert)}\int_{\mathbb R^n}\big\vert e^{-\vert\xi\vert^2/2}\big\vert^2\ \mathrm d\xi \\[5pt]
	& = e^{-2TF(\vert\xi_0\vert)}\pi^{n/2} >\varepsilon\pi^{n/2}. 
\end{align*}
The parameter $l\gg1$ can therefore be adjusted so that \eqref{06102020E1} holds. The values of $\xi_0\in\mathbb R^n$ and $l\gg1$ are now fixed. We therefore deduce from \eqref{19122019E4} and \eqref{06102020E1} that
\begin{equation}\label{16082020E1}
	M_{l,\xi_0}\le C\int_0^T\big\Vert e^{-tF(\vert D_x\vert)}g_{l,\xi_0}\big\Vert^2_{L^2(\omega)}\ \mathrm dt\quad\text{with}\quad M_{l,\xi_0} = \delta_{l,\xi_0} - \varepsilon\Vert g_{l,\xi_0}\Vert^2_{L^2(\mathbb R^n)}>0.
\end{equation}
Moreover, by introducing $\mathscr F_{\xi}^{-1}$ the partial inverse Fourier transform with respect to the variable $\xi\in\mathbb R^n$ and using \eqref{06112020E2}, the right-hand side of this inequality (up to the constant $C$) writes as 
\begin{align*}
	\int_0^T\big\Vert e^{-tF(\vert D_x\vert)}g_{l,\xi_0}\big\Vert^2_{L^2(\omega)}\ \mathrm dt 
	& = (2\pi)^n \int_0^T\int_{\omega}\big\vert\mathscr F^{-1}_{\xi}(e^{-ix_0\cdot(\xi-\xi_0)}e^{-tF(\vert\xi\vert)}e^{-l^2\vert\xi-\xi_0\vert^2/2})(x)\big\vert^2\ \mathrm dx\mathrm dt \\[5pt]
	& = (2\pi)^n \int_0^T\int_{\omega}\big\vert\mathscr F^{-1}_{\xi}(e^{-tF(\vert\xi\vert)}e^{-l^2\vert\xi-\xi_0\vert^2/2})(x-x_0)\big\vert^2\ \mathrm dx\mathrm dt \\[5pt]
	& = (2\pi)^n \int_0^T\int_{\omega-x_0}\big\vert\mathscr F^{-1}_{\xi}(e^{-tF(\vert\xi\vert)}e^{-l^2\vert\xi-\xi_0\vert^2/2})(x)\big\vert^2\ \mathrm dx\mathrm dt.
\end{align*}
Given $r>0$ a positive radius whose value will be chosen later, we split the previous integral in two parts and obtain the following estimate:
\begin{multline}\label{23082018E6}
	\int_0^T\big\Vert e^{-tF(\vert D_x\vert)}g_{l,\xi_0}\big\Vert^2_{L^2(\omega)}\ \mathrm dt \\[5pt]
	\le (2\pi)^n\int_0^T\int_{(\omega-x_0)\cap[-r,r]^n}\big\vert\mathscr F^{-1}_{\xi}(e^{-tF(\vert\xi\vert)}e^{-l^2\vert\xi-\xi_0\vert^2/2})(x)\big\vert^2\ \mathrm dx\mathrm dt \\[5pt]
	+ (2\pi)^n\int_0^T\int_{\vert x\vert>r}\big\vert\mathscr F^{-1}_{\xi}(e^{-tF(\vert\xi\vert)}e^{-l^2\vert\xi-\xi_0\vert^2/2})(x)\big\vert^2\ \mathrm dx\mathrm dt.
\end{multline}
Now, we study one by one the two integrals appearing in the right-hand side of \eqref{23082018E6}. First, notice that for all $0\le t\le T$,
\begin{align*}
	\big\Vert\mathscr F^{-1}_{\xi}(e^{-tF(\vert\xi\vert)}e^{-l^2\vert\xi-\xi_0\vert^2/2})\big\Vert_{L^{\infty}(\mathbb R^n)}
	& \le\frac1{(2\pi)^n}\big\Vert e^{-tF(\vert\xi\vert)}e^{-l^2\vert\xi-\xi_0\vert^2/2}\big\Vert_{L^1(\mathbb R^n)} \\[5pt]
	& \le\frac{e^{-T\inf F}}{(2\pi)^n}\big\Vert e^{-l^2\vert\xi-\xi_0\vert^2/2}\big\Vert_{L^1(\mathbb R^n)} \\[5pt]
	& = \frac{e^{-T\inf F}}{(2\pi)^n}\bigg(\frac{2\pi}{l^2}\bigg)^{n/2}.
\end{align*}
It therefore follows from the invariance by translation of the Lebesgue measure that
\begin{multline}\label{23082018E7}
	(2\pi)^n\int_0^T\int_{(\omega-x_0)\cap[-r,r]^n}\big\vert\mathscr F^{-1}_{\xi}(e^{-tF(\vert\xi\vert)}e^{-l^2\vert\xi-\xi_0\vert^2/2})(x)\big\vert^2\ \mathrm dx\mathrm dt  \\[5pt]
	\le \frac{e^{-2T\inf F}}{l^{2n}}\int_0^T\Leb\big((\omega-x_0)\cap [-r,r]^n\big)\ \mathrm dt
	= \frac{Te^{-2T\inf F}}{l^{2n}}\Leb\big(\omega\cap(x_0+[-r,r]^n)\big).
\end{multline}
In order to control the second integral, we use the dominated convergence theorem which justifies the following convergence
$$\int_0^T\int_{\vert x\vert>r}\big\vert\mathscr F^{-1}_{\xi}(e^{-tF(\vert\xi\vert)}e^{-l^2\vert\xi-\xi_0\vert^2/2})(x)\big\vert^2\ \mathrm dx\mathrm dt \underset{r\rightarrow+\infty}{\rightarrow}0,$$
since
$$\mathscr F^{-1}_{\xi}(e^{-tF(\vert\xi\vert)}e^{-l^2\vert\xi-\xi_0\vert^2/2})\in L^2([0,T]\times\mathbb R^n).$$
Thus, we can choose the radius $r\gg1$ large enough so that
\begin{equation}\label{19122019E6}
	(2\pi)^nC\int_0^T\int_{\vert x\vert>r}\big\vert\mathscr F^{-1}_{\xi}(e^{-tF(\vert\xi\vert)}e^{-l^2\vert\xi-\xi_0\vert^2/2})(x)\big\vert^2\ \mathrm dx\mathrm dt\le \frac{M_{l,\xi_0}}2.
\end{equation}
Gathering \eqref{16082020E1}, \eqref{23082018E6}, \eqref{23082018E7} and \eqref{19122019E6}, we obtain the following estimate
$$\forall x_0\in\mathbb R^n,\quad \frac{M_{l,\xi_0}}2\le\frac C{l^{2n}}Te^{-2T\inf F}\Leb\big(\omega\cap(x_0+[-r,r]^n)\big).$$
This proves that the control support $\omega$ is actually a thick set.

\subsection{Proof of Theorem \ref{11112020T1}: assertion $(ii)$} In this second subsection, assuming that $\omega$ is a thick set and that $\liminf_{+\infty}F>\inf F$, we establish the estimate \eqref{29102020T1}. The strategy consists in constructing an adapted Lyapunov function, inspired by the proof of \cite{Xiang} (Theorem 1.1) in which the author studies the stabilization properties of the heat equation posed on bounded domains of $\mathbb R^n$.

We consider the function $G = F - \inf F$. Since $\liminf_{+\infty} G >0$ by assumption, there exists a positive constant $R_0>0$ such that 
\begin{equation}\label{04032021E1}
	\forall R \geq R_0, \quad \tilde\alpha_R = \inf_{r\geq R} G(r) >0.
\end{equation}
Let us fix $R\geq R_0$ and consider two positive real numbers $\lambda_R,\mu_R>0$ to be chosen later. For all initial data $f_0 \in L^2(\mathbb R^n)$ and all time $t\geq0$, we define $f$ as the mild solution on $[0,+\infty)$ of the control system (\hyperref[02062020E1]{$E_G$}) with feedback $\lambda_RK_R$ at time $t$, that is,
\begin{equation}\label{04032021E2}
	\forall t\geq0,\quad f(t)= e^{-t(G(\vert D_x\vert)+ \lambda_R\mathbbm 1_{\omega}K_R)}f_0,
\end{equation}
where $K_R$ stands for the following orthogonal projection
\begin{equation}\label{06202020E3}
	K_R : L^2(\mathbb R^n)\rightarrow\big\{g\in L^2(\mathbb R^n) : \Supp\widehat g\subset\overline{B(0,R)}\big\}.
\end{equation}
We want to prove that the two constants $\lambda_R$ and $\mu_R$ can be adjusted so that the following estimate holds for all $t\geq0$,
\begin{equation}\label{06112020E4}
	\Vert f(t)\Vert^2_{L^2(\mathbb R^n)}\le\mu_Re^{-\tilde\alpha_Rt}\Vert f_0\Vert^2_{L^2(\mathbb R^n)}.
\end{equation}
To that end, we consider the following Lyapunov function
\begin{equation}\label{Lyapunov}
	V(y) = \mu_R \Vert K_R y\Vert^2_{L^2(\mathbb R^n)} + \Vert(1-K_R)y\Vert^2_{L^2(\mathbb R^n)},\quad y\in L^2(\mathbb R^n).
\end{equation}
Let us equip the operator $G(\vert D_x\vert)+ \lambda_R\mathbbm 1_{\omega}K_R$ with the following domain 
$$D_G = \big\{u\in L^2(\mathbb R^n) : G(\vert D_x\vert)u\in L^2(\mathbb R^n)\big\}.$$
From now, we assume that the Fourier transform of the initial datum $f_0$ is compactly supported. As a consequence, $f_0\in D_G$ and the function $V\circ f$ is of class $C^1$ on $(0,+\infty)$, with
$$\forall t>0,\quad\frac{\mathrm d}{\mathrm dt} V(f(t)) = \mu_R\frac{\mathrm d}{\mathrm dt}\Vert K_Rf(t)\Vert^2_{L^2(\mathbb R^n)} + \frac{\mathrm d}{\mathrm dt}\Vert(1-K_R)f(t)\Vert^2_{L^2(\mathbb R^n)}.$$
We shall need to estimate each term of the right-hand side of this equality. On the one hand, noticing that the operators $G(\vert D_x\vert)$ and $K_R$ commute (they are Fourier multipliers), we have that for all $t>0$,
\begin{multline*}
	\mu_R\frac{\mathrm d}{\mathrm dt}\Vert K_Rf(t)\Vert^2_{L^2(\mathbb R^n)} = 2\mu_R\Reelle\langle K_Rf(t),K_Rf'(t)\rangle_{L^2(\mathbb R^n)} \\
	= -2\mu_R\langle K_Rf(t),G(\vert D_x\vert)K_Rf(t)\rangle_{L^2(\mathbb R^n)} -2 \lambda_R\mu_R\Vert K_Rf(t)\Vert^2_{L^2(\omega)}.
\end{multline*}
The operator $G(|D_x|)$ being accretive, we get that for all $t>0$,
$$\langle K_Rf(t),G(\vert D_x\vert)K_Rf(t)\rangle_{L^2(\mathbb R^n)}\geq0,$$
and as a consequence,
$$\mu_R\frac{\mathrm d}{\mathrm dt}\Vert K_Rf(t)\Vert^2_{L^2(\mathbb R^n)}\le-2\lambda_R\mu_R\Vert K_Rf(t)\Vert^2_{L^2(\omega)}.$$
Moreover, O. Kovrijkine established in \cite{Kovrijkine} (Theorem~3) a quantitative version of the Logvinenko-Sereda theorem for thick sets which will allow us to control the right-hand side of the above estimate. Precisely, this result is the following:

\begin{thm}[Theorem 3 in \cite{Kovrijkine}]\label{Kovrij} There exists a universal positive constant $C_n\geq e$ depending only on the dimension $n\geq 1$ such that for all $\gamma$-thick at scale $L>0$ subset $\omega\subset\mathbb R^n$, 
$$\forall R>0, \forall f\in L^2(\mathbb R^n),\ \Supp\widehat f\subset[-R,R]^n,\quad\Vert f\Vert_{L^2(\mathbb R^n)}\le\Big(\frac{C_n}{\gamma}\Big)^{C_n(1+LR)}\Vert f\Vert_{L^2(\omega)}.$$
\end{thm}

\noindent We therefore deduce from this theorem and the definition \eqref{06202020E3} of the orthogonal projection $K_R$ that there exists a positive constant $C=C(\omega)\geq1$ only depending on the thick set $\omega$ (and not on the positive real number $R$) such that for all $t>0$,
\begin{equation}\label{first_term}
	\mu_R\frac{\mathrm d}{\mathrm dt}\Vert K_Rf(t)\Vert^2_{L^2(\mathbb R^n)}\le-2\lambda_R\mu_RC^{-1}e^{-CR}\Vert K_Rf(t)\Vert^2_{L^2(\mathbb R^n)}.
\end{equation}
On the other hand, recalling that the operators $G(\vert D_x\vert)$ and $K_R$ commute, the second term we aim at controlling is given for all $t>0$ by
\begin{multline}\label{second_term}
	\frac{\mathrm d}{\mathrm dt}\Vert(1-K_R)f(t)\Vert^2_{L^2(\mathbb R^n)} = 2\Reelle\langle(1-K_R)f(t),(1-K_R)f'(t)\rangle_{L^2(\mathbb R^n)} \\[5pt]
	= -2\langle(1-K_R)f(t),G(\vert D_x\vert)(1-K_R)f(t)\rangle_{L^2(\mathbb R^n)} -2\lambda_R\Reelle\langle(1-K_R)f(t),(1-K_R)\mathbbm 1_{\omega}K_Rf(t)\rangle_{L^2(\mathbb R^n)}.
\end{multline}
We notice that by definition of the orthogonal projection $K_R$, the Fourier transforms of the functions $(1-K_R)f(t)$ are supported in $\mathbb R^n\setminus B(0,R)$, which implies that for all $t>0$,
\begin{equation}\label{second_term1}
	2\langle(1-K_R)f(t),G(\vert D_x\vert)(1-K_R)f(t)\rangle_{L^2(\mathbb R^n)}\geq2\tilde\alpha_R\Vert(1-K_R)f(t)\Vert^2_{L^2(\mathbb R^n)}.
\end{equation}
By using in addition Cauchy-Schwarz' and Young's inequalities, we obtain
\begin{multline}\label{second_term2}
	-2\lambda_R\Reelle\langle(1-K_R)f(t),(1-K_R)\mathbbm 1_{\omega}K_Rf(t)\rangle_{L^2(\mathbb R^n)}\le2\lambda_R\Vert(1-K_R)f(t)\Vert_{L^2(\mathbb R^n)}\Vert K_Rf(t)\Vert_{L^2(\mathbb R^n)} \\
	\leq\frac{\lambda_R^2}{\tilde\alpha_R}\Vert K_R f(t)\Vert^2_{L^2(\mathbb R^n)} + \tilde\alpha_R\Vert(1-K_R)f(t)\Vert^2_{L^2(\mathbb R^n)}.
\end{multline}
Combining the estimates \eqref{second_term}, \eqref{second_term1} and \eqref{second_term2}, we obtain that for all $t>0$,
$$\frac{\mathrm d}{\mathrm dt}\Vert(1-K_R)f(t)\Vert^2_{L^2(\mathbb R^n)}\le\frac{\lambda_R^2}{\tilde\alpha_R}\Vert K_R f(t)\Vert^2_{L^2(\mathbb R^n)} - \tilde\alpha_R\Vert(1-K_R)f(t)\Vert^2_{L^2(\mathbb R^n)}.$$
This inequality and \eqref{first_term} then imply that for all $t>0$,
$$\frac{\mathrm d}{\mathrm dt}V(f(t))\leq-2\bigg(\lambda_R\mu_RC^{-1}e^{-CR}-\frac{\lambda_R^2}{\tilde\alpha_R}\bigg)\big\Vert K_Rf(t)\big\Vert^2_{L^2(\mathbb R^n)} - \tilde\alpha_R \big\Vert(1-K_R)f(t)\big\Vert^2_{L^2(\mathbb R^n)}.$$
By making the following choices for the constants $\mu_R$ and $\lambda_R$,
\begin{equation*}
	\mu_R= 2C^2 e^{2CR},\quad\lambda_R = C e^{CR}\tilde\alpha_R,
\end{equation*}
we get that for all $t>0$,
$$\frac{\mathrm d}{\mathrm dt}V(f(t))\le-\tilde\alpha_RV(f(t)).$$
This latest estimate and Gr\"onwall's inequality readily imply that for all $t\geq0$,
$$V(f(t))\le e^{-\tilde\alpha_R t}V(f(0)),$$
and then, by Pythagore's theorem, since $\mu_R \geq 1$, we obtain that for all $t\geq0$,
$$\Vert f(t)\Vert^2_{L^2(\mathbb R^n)}\le\mu_Re^{-\tilde\alpha_R t}\Vert f_0\Vert^2_{L^2(\mathbb R^n)}.$$
Since the set of functions of $L^2(\mathbb R^n)$ with compactly supported Fourier transforms is dense in $L^2(\mathbb R^n)$, and that the evolution operators $e^{-t(G(\vert D_x\vert)+ \lambda_R\mathbbm 1_{\omega}K_R)}$ are continuous on this space, the above estimate is valid for all $f_0\in L^2(\mathbb R^n)$. The estimate \eqref{06112020E4} therefore holds. Recalling the definitions \eqref{04032021E1} and \eqref{04032021E2} of $\tilde\alpha_R$ and $f(t)$ respectively, and also recalling that $G = F-\inf F$, we have established that for all $R\geq R_0$ and $t\geq0$,
$$\big\Vert e^{-t(F(\vert D_x\vert)+Ce^{CR}(\alpha_R-\inf F)\mathbbm 1_{\omega}K_R)}\big\Vert^2_{\mathcal L(L^2(\mathbb R^n))}\le 2C^2e^{2CR}e^{-(\alpha_R+\inf F)t},$$
with $\alpha_R = \inf_{r\geq R} F(r)$. This ends the proof of assertion $(ii)$ in Theorem \ref{11112020T1}.

\subsection{Proof of Proposition~\ref{negativ_result}} In this last subsection, we prove Proposition~\ref{negativ_result} which provides a negative result for the rapid stabilization of the evolution equation \eqref{01062020E1}. We assume that $\omega$ is not dense in $\mathbb R^n$, and also that $\lim_{+\infty}F$ exists and is a non-negative real number $L\geq0$. Since the function $F$ is continuous, this implies that $F$ is bounded, i.e. $\sup F<+\infty$. We aim at proving that if $\alpha>\sup F$, then the equation \eqref{01062020E1} is not exponentially stabilizable from $\omega$ at rate $\alpha$. To that end, we will use the following interpretation of exponential stabilization at rate $\alpha>0$ in terms of observability.

\begin{prop}[Theorem 1.1 in \cite{LWXY}]\label{18122020T2} If the evolution system \eqref{01062020E1} is exponentially stabilizable from $\omega$ at rate $\alpha >0$, then there exists a positive constant $A_{\alpha}>0$ such that for all $T>0$, there exists a positive constant $C_{\alpha, T} >0$ satisfying that for all $g\in L^2(\mathbb R^n)$,
$$\big\Vert e^{-TF(\vert D_x\vert)}g\big\Vert^2_{L^2(\mathbb R^n)}\le C_{\alpha,T} \int_0^T\big\Vert e^{-tF(\vert D_x\vert)}g\big\Vert^2_{L^2(\omega)}\ \mathrm dt 
+ A_{\alpha} e^{-2\alpha T}\Vert g\Vert^2_{L^2(\mathbb R^n)}.$$
\end{prop}

The proof of Proposition~\ref{18122020T2} is contained in the proof of Theorem 1.1 in \cite{LWXY}, although \cite{LWXY} (Theorem 1.1) only states characterizations of complete stabilization. For the sake of completeness, we recall the arguments given by the authors of \cite{LWXY} in Section \ref{appendix}. 

Proceeding by contradiction, we consider $\alpha>\sup F$ and assume that the equation \eqref{01062020E1} is exponentially stabilizable at rate $\alpha$ from $\omega$. According to Proposition~\ref{18122020T2}, there exists a positive constant $A_{\alpha}>0$ such that for all $T>0$, there exists a positive constant $C_{\alpha, T} >0$ such that for all $g\in L^2(\mathbb R^n)$,
\begin{equation}\label{observability_rapid_stab}
	\big\Vert e^{-TF(\vert D_x\vert)}g\big\Vert^2_{L^2(\mathbb R^n)}\le C_{\alpha,T} \int_0^T\big\Vert e^{-tF(\vert D_x\vert)}g\big\Vert^2_{L^2(\omega)}\ \mathrm dt 
	+ A_{\alpha} e^{-2\alpha T}\Vert g\Vert^2_{L^2(\mathbb R^n)}.
\end{equation}
Since we get that for all $T>0$ and $g \in L^2(\mathbb R^n)$,
$$\big\Vert e^{-TF(\vert D_x\vert)}g\big\Vert_{L^2(\mathbb R^n)} \geq e^{-T \sup F} \Vert g\Vert_{L^2(\mathbb R^n)},$$
it follows from \eqref{observability_rapid_stab} that for all $T>0$ and $g \in L^2(\mathbb R^n)$,
\begin{equation}\label{observability_rapid_stab2}
	\Vert g\Vert^2_{L^2(\mathbb R^n)}\le C_{\alpha,T}e^{2T\sup F} \int_0^T\big\Vert e^{-tF(\vert D_x\vert)}g\big\Vert^2_{L^2(\omega)}\ \mathrm dt 
	+ A_{\alpha} e^{2(\sup F-\alpha) T}\Vert g\Vert^2_{L^2(\mathbb R^n)}.
\end{equation}
Notice that since $\sup F - \alpha <0$, we have $\lim_{T \to +\infty} e^{2(\sup F-\alpha) T}=0$. Therefore, there exists $T_0>0$ such that $A_{\alpha}e^{2(\sup F-\alpha)T_0}<1/2$. This fact, together with \eqref{observability_rapid_stab2}, imply that for all $g \in L^2(\mathbb R^n)$,
\begin{equation}\label{observability_rapid_stab3}
	\Vert g\Vert^2_{L^2(\mathbb R^n)}\le 2C_{\alpha,T_0}e^{2T_0\sup F}\int_0^{T_0}\big\Vert e^{-tF(\vert D_x\vert)}g\big\Vert^2_{L^2(\omega)}\ \mathrm dt.
\end{equation}
Since $\omega$ is not dense in $\mathbb R^n$ and that the evolution equation \eqref{01062020E1} is invariant under translations, we can assume that there exists a positive radius $r>0$ such that $B(0, r) \subset \mathbb R^n \setminus \omega$. Let us fix a non-zero $L^2$-function $\psi$ and define the function $g_h$ for all $h >0$ by
$$\forall x \in \mathbb R^n, \quad g_h(x) = \psi\Big(\frac xh\Big).$$
On the one hand, we have that for all $h>0$,
\begin{equation}\label{gh1}
	\Vert g_h\Vert^2_{L^2(\mathbb R^n)} = h^n \Vert\psi\Vert^2_{L^2(\mathbb R^n)}.
\end{equation}
On the other hand, we get that for all $h>0$,
\begin{align}\label{gh2}
	\int_0^{T_0}\big\Vert e^{-tF(\vert D_x\vert)}g_h\big\Vert^2_{L^2(\omega)}\ \mathrm dt 
	& \le \int_0^{T_0}\big\Vert e^{-tF(\vert D_x\vert)}g_h\big\Vert^2_{L^2(\mathbb R^n \setminus B(0,r))}\ \mathrm dt \\[5pt]
	& = h^n\int_0^{T_0}\big\Vert e^{-tF(\vert D_x\vert/h)}\psi\big\Vert^2_{L^2(\mathbb R^n \setminus B(0,r/h))}\ \mathrm dt. \nonumber
\end{align}
It follows from \eqref{observability_rapid_stab3}, \eqref{gh1} and \eqref{gh2} that for all $h>0$,
$$\Vert\psi\Vert^2_{L^2(\mathbb R^n)}\le2C_{\alpha,T_0}e^{2T_0\sup F}\int_0^{T_0}\big\Vert e^{-tF(\vert D_x\vert/h)}\psi\big\Vert^2_{L^2(\mathbb R^n \setminus B(0,r/h))}\ \mathrm dt.$$
Moreover, since $L=\lim_{+\infty}F\geq0$, we get
\begin{align*}
	&\ \int_0^{T_0}\big\Vert e^{-tF(\vert D_x\vert/h)}\psi\big\Vert^2_{L^2(\mathbb R^n \setminus B(0,r/h))}\ \mathrm dt \\
	\le &\ 2\int_0^{T_0}\big\Vert e^{-tF(\vert D_x\vert/h)}\psi - e^{-tL}\psi\big\Vert^2_{L^2(\mathbb R^n \setminus B(0,r/h))}\ \mathrm dt + 2\int_0^{T_0}\big\Vert e^{-tL}\psi\big\Vert^2_{L^2(\mathbb R^n \setminus B(0,r/h))}\ \mathrm dt \\
	\le &\ 2\int_0^{T_0}\big\Vert e^{-tF(\vert D_x\vert/h)}\psi - e^{-tL}\psi\big\Vert^2_{L^2(\mathbb R^n)}\ \mathrm dt + 2T_0\Vert\psi\Vert^2_{L^2(\mathbb R^n \setminus B(0,r/h))},
\end{align*}
and this implies that
\begin{multline*}
	\Vert\psi\Vert^2_{L^2(\mathbb R^n)} \leq 4C_{\alpha,T_0}e^{2T_0\sup F} \int_0^{T_0} \big\Vert e^{-tF(\vert D_x\vert/h)}\psi - e^{-tL}\psi\big\Vert^2_{L^2(\mathbb R^n)}\ \mathrm dt \\
	+ 4C_{\alpha,T_0}T_0e^{2T_0\sup F}\Vert\psi\Vert^2_{L^2(\mathbb R^n \setminus B(0,r/h))}.
\end{multline*}
To obtain a contradiction, it remains to check that each term of the right-hand side of the above inequality converges to $0$ as $h\rightarrow0^+$. For the first term, we deduce from Plancherel's theorem and the dominated convergence theorem (recall that the function $F$ is bounded from below) that
$$\int_0^{T_0} \big\Vert e^{-tF(\vert D_x\vert/h)} \psi - e^{-tL}\psi \big\Vert^2_{L^2(\mathbb R^n)}\ \mathrm dt
= \frac1{(2\pi)^n}\int_0^{T_0}\big\Vert e^{-tF(\vert\xi\vert/h)}\widehat{\psi} - e^{-tL}\widehat{\psi}\big\Vert^2_{L^2(\mathbb R^n)}\ \mathrm dt \underset{h\rightarrow0}{\rightarrow}0.$$
For the second one, the dominated convergence theorem readily implies that
$$\Vert\psi\Vert^2_{L^2(\mathbb R^n \setminus B(0,r/h))} \underset{h\rightarrow0}{\rightarrow} 0.$$
This concludes the proof of Proposition~\ref{negativ_result}.

\section{Quasi-analytic sequences}\label{QA}

This section is devoted to recall some properties of quasi-analytic sequences. 

\subsection{Notion of quasi-analytic sequence} Let us first recall that a sequence $(M_k)_{k\geq0}$ of non-negative real numbers is said to be \textit{log-convex} when it satisfies
\begin{equation}\label{log_conv}
	\forall k\geq1,\quad M_k^2 \le M_{k+1} M_{k-1}.
\end{equation}
A relevant example of log-convex sequence is the sequence $\mathcal M^F$ associated with any continuous function $F : [0, +\infty)\rightarrow\mathbb R$ bounded from below. Let us recall that the elements $M^F_k$ of the sequence $\mathcal M^F$ are assumed to be positive real numbers and defined for all $k\geq0$ by 
\begin{equation}\label{ex_qa_sequence}
	0<M^F_k = \sup_{r\geq0} r^ke^{-F(r)}<+\infty.
\end{equation}
Let us also recall that a sequence $\mathcal M=(M_k)_{k\geq0}$ of positive real numbers is called quasi-analytic when for all real numbers $a<b$, the associated Denjoy-Carleman class
$$\mathcal C_{\mathcal M}([a,b])= \big\{ f \in\mathcal C^{\infty}([a,b],\mathbb C) : \forall k\geq0,\forall x\in[a,b],\ \vert f^{(k)}(x)\vert \le M_k\big\},$$
is quasi-analytic, meaning that any function belonging to this class is identically equal to zero when satisfying
$$\exists x_0\in[a,b],\forall k\geq0,\quad f^{(k)}(x_0) = 0.$$
There exist several necessary and sufficient equivalent conditions that ensure a log-convex sequence to be quasi-analytic. In this work, we will make use of the following one, which is a particular case of Denjoy-Carleman's theorem, see e.g. \cite{Koosis} (Theorem p. 91).

\begin{thm}[Denjoy-Carleman]\label{Den_Carl_thm} A log-convex sequence $(M_k)_{k\geq0}$ of positive real numbers is quasi-analytic if and only if 
$$\sum_{k=0}^{+\infty}\frac{M_k}{M_{k+1}} = + \infty.$$
\end{thm}

For all connected open set $U\subset\mathbb R^n$, let us consider the more general Denjoy-Carleman class associated with a sequence $\mathcal M$ of positive real numbers:
$$\mathcal C_{\mathcal M}(U)= \big\{ f \in\mathcal C^{\infty}(U,\mathbb C) : \forall \beta\in\mathbb N^n,\ \Vert\partial_x^{\beta}f\Vert_{L^{\infty}(U)} \le  M_{\vert\beta\vert}\big\}.$$
We check in the following proposition that when the sequence $\mathcal M$ is quasi-analytic, then all the associated classes $C_{\mathcal M}(U)$ are also quasi-analytic.

\begin{prop}\label{23122020P1} Let $U\subset\mathbb R^n$ be a connected open set and $\mathcal M$ be a quasi-analytic sequence of positive real numbers. Then, the associated class $C_{\mathcal M}(U)$ is quasi-analytic, meaning that any function $f$ belonging to this class that vanishes at infinite order in a point of the set $U$ is identically equal to zero.
\end{prop}

\begin{proof} This proof is based on elementary topology arguments. Let $f\in C_{\mathcal M}(U)$ vanishing at infinite order in a point of the set $U$, that is,
$$\exists x_0\in U,\forall\beta\in\mathbb N^n,\quad \partial^{\beta}_xf(x_0) = 0.$$
We set 
\begin{equation}\label{17122020E1}
	\Omega = \big\{x\in U : \forall\beta\in\mathbb N^n,\ \partial^{\beta}_xf(x) = 0\big\}.
\end{equation}
The set $\Omega$ is non-empty and closed in $U$, since $f$ is a smooth function. Let us check that this is an open subset of $U$. Since $U$ is connected this will imply that $\Omega = U$ and therefore that the function $f$ is identically equal to zero. Let $x\in\Omega$. Since $U$ is open, there exists a radius $r>0$ such that $\overline{B(x,r)}\subset U$. Fixing $y\in B(x,r)$, we consider the function $g_y:[-2r,2r]\rightarrow\mathbb C$ defined by
$$g_y(t) = f\bigg(x + \frac{t(y-x)}{2r}\bigg),\quad t\in[-2r,2r].$$
Since $f\in C_{\mathcal M}(U)$, we get that $g_y\in C_{\mathcal M}([-2r,2r])$. Indeed, it follows from a direct computation that for all $k\geq0$ and $-2r\le t\le 2r$,
\begin{align*}
	\big\vert g_y^{(k)}(t)\big\vert & = \bigg\vert\mathrm d^kf\bigg(x + \frac{t(y-x)}{2r}\bigg)\cdot\bigg(\frac{y-x}{2r},\ldots,\frac{y-x}{2r}\bigg)\bigg\vert \\
	& \le \sum_{1\le i_1,\ldots,i_k\le n}\bigg\vert\frac{\partial^kf}{\partial_{x_{i_1}}\ldots\partial_{x_{i_k}}}\bigg(x + \frac{t(y-x)}{2r}\bigg)\bigg\vert\bigg\vert\bigg(\frac{y-x}{2r}\bigg)_{i_1}\bigg\vert\ldots\bigg\vert\bigg(\frac{y-x}{2r}\bigg)_{i_k}\bigg\vert \\[5pt]
	& \le\bigg(\sum_{1\le i_1,\ldots,i_k\le n}\frac1{2^{i_1+\ldots+i_k}}\bigg)M_k \le\bigg(\sum_{j=1}^{+\infty}\frac1{2^j}\bigg)^kM_k =  M_k.
\end{align*}
Moreover, since $x\in\Omega$, we deduce from the definition \eqref{17122020E1} of the set $\Omega$ that 
$$\forall k\geq0,\quad g^{(k)}_y(0) = 0.$$ 
The sequence $\mathcal M$ being quasi-analytic by assumption, this implies that $g$ is identically equal to zero on $[-2r,2r]$, and so $f(y) = g(2r) = 0$. The set $\Omega$ is then open in $U$.
\end{proof}

We are now in position to state the main proposition of this subsection, which is instrumental in the proof of Theorem~\ref{21102020T1}. This result is established in the work \cite{JM}.

\begin{prop}[Corollary 2.8  in \cite{JM}]\label{15122020P1} Let $\mathcal M$ be a log-convex quasi-analytic sequence, $t>0$ be a positive real number and $\omega \subset (0,1)^n$ be a measurable set. If $\Leb(\omega)\geq\lambda$ with $0< \lambda\leq1$, then there exists a positive constant $C(\lambda, t, \mathcal M, n) >0$ such that for all $f\in\mathcal C_{\mathcal M}((0,1)^n)$ satisfying $\Vert f\Vert_{L^{\infty}((0,1)^n)}\geq t$,
$$\Vert f \Vert_{L^2((0,1)^n)} \le C(\lambda, t, \mathcal M, n)\Vert f\Vert_{L^2(\omega)}.$$
\end{prop}

A quantitative version of this result (with $L^{\infty}$ norms) is stated by Nazarov, Sodin and Volberg in \cite{NSV} and this stronger result provides an explicit dependence of $C(\lambda, t, \mathcal M, n)$ with respect to $\lambda$, $t$ and $\mathcal M$ when $n=1$.

\subsection{Examples} Now, we establish useful results to construct quasi-analytic sequences and we provide examples of such sequences. Let us begin with a straightforward general lemma.

\begin{lem}\label{17122020L1} Let $\mathcal M$ and $\mathcal M'$ be two sequences of positive real numbers satisfying $\mathcal M\le\mathcal M'$. Then, if $\mathcal M'$ is a quasi-analytic sequence, so is the sequence $\mathcal M$.
\end{lem}

\begin{proof} Since $\mathcal M\le \mathcal M'$ by assumption, we get that $C_{\mathcal M}([a,b])\subset C_{\mathcal M'}([a,b])$ for all real numbers $a<b$ and the result follows.
\end{proof}

We now focus on the log-convex sequences of the form $\mathcal M^F$ defined in \eqref{ex_qa_sequence}, which are the ones considered in this work. First, we study the stability of the quasi-analyticity property of the sequence $\mathcal M^F$ when modifying the function $F$.


\begin{lem}\label{17122020L2} If the sequence $\mathcal M^F$ associated with a continuous function $F : [0,+\infty)\rightarrow\mathbb R$ bounded from below is quasi-analytic, then for all $c\in\mathbb R$ and $T>0$, the sequence $\mathcal M^{TF+c}$ associated with the continuous function $TF+c$ is also quasi-analytic.
\end{lem}

\begin{proof} By considering the function $G = F+c/T$, whose associated sequence $\mathcal M^G$ is immediately quasi-analytic, we can assume that $c=0$. Let us consider a positive integer $p\geq1$ so that $T\geq1/p>0$. We therefore have that for all $k\geq0$,
\begin{align*}
	M_k^{TF} = \sup_{r\geq0} r^ke^{-TF(r)}
	& \le e^{(\frac1p-T)\inf F}\sup_{r\geq0} r^ke^{-\frac1pF(r)} \\
	& = e^{(\frac1p-T)\inf F}\big(\sup_{r\geq0} r^{kp}e^{-F(r)}\big)^{\frac1p} \\
	& = e^{(\frac1p-T)\inf F}(M^F_{kp})^{\frac1p}=: M_k.
\end{align*}
According to Lemma \ref{17122020L1}, it is sufficient to check that the sequence $(M_k)_k$ is quasi-analytic to end the proof of Lemma \ref{17122020L2}. To that end, we will use Denjoy-Carleman's theorem, that is Theorem \ref{Den_Carl_thm}, and prove that
$$\sum_{k=0}^{+\infty}\frac{M_k}{M_{k+1}} = +\infty.$$
Notice that since the sequence $\mathcal M^F$ is log-convex, the sequence $(M^F_k/M^F_{k+1})_k$ is non-increasing. As a consequence, we get that
$$\sum_{k=0}^{+\infty}\frac{M_k}{M_{k+1}} = \sum_{k=0}^{+\infty}\bigg(\frac{M^F_{kp}}{M^F_{kp+p}}\bigg)^{\frac1p}
= \sum_{k=0}^{+\infty}\bigg(\prod_{j=0}^{p-1}\frac{M^F_{kp+j}}{M^F_{kp+j+1}}\bigg)^{\frac1p}
\geq \sum_{k=0}^{+\infty}\frac{M^F_{kp+p-1}}{M^F_{kp+p}}.$$
It is therefore sufficient to prove that 
\begin{equation}\label{19122020E1}
	\sum_{k=0}^{+\infty}\frac{M^F_{kp+p-1}}{M^F_{kp+p}} = +\infty.
\end{equation}
By using again that the sequence $(M^F_k/M^F_{k+1})_k$ is non-increasing, we obtain by also applying an Euclidean division by the positive integer $p\geq1$ that
$$\sum_{k=0}^{+\infty}\frac{M^F_k}{M^F_{k+1}} = \sum_{r=0}^{p-1}\sum_{k=0}^{+\infty}\frac{M^F_{kp+r}}{M^F_{kp+r+1}}
\le p\sum_{k=0}^{+\infty}\frac{M^F_{kp}}{M^F_{kp+1}}.$$
Yet, the log-convex sequence $\mathcal M^F$ is quasi-analytic by assumption, and Theorem \ref{Den_Carl_thm} implies that
\begin{equation}\label{9122020E2}
	\sum_{k=0}^{+\infty}\frac{M^F_k}{M^F_{k+1}} = +\infty\quad\text{and so}\quad\sum_{k=0}^{+\infty}\frac{M^F_{kp}}{M^F_{kp+1}} = +\infty.
\end{equation}
Finally, exploiting a last time the non-increasing property of the sequence $(M^F_k/M^F_{k+1})_k$, we obtain that
$$\sum_{k=0}^{+\infty}\frac{M^F_{(k+1)p}}{M^F_{(k+1)p+1}}\le\sum_{k=0}^{+\infty}\frac{M^F_{(k+1)p-1}}{M^F_{(k+1)p}}.$$
This inequality and \eqref{9122020E2} show that \eqref{19122020E1} holds. This ends the proof of Lemma \ref{17122020L2}.
\end{proof}

It is a very interesting problem to characterize the functions $F$ that generate quasi-analytic sequences $\mathcal M^F$. This question has been addressed by B. Jaye and M. Mitkovski in \cite{JM}, where these authors provided a necessary and sufficient condition on some functions $F$ that ensures the associated sequence $\mathcal M^F$ to be quasi-analytic, by exploiting Denjoy-Carleman's theorem.

\begin{prop}[\cite{JM}]\label{06112020P1} Let $F : [0, +\infty)\rightarrow\mathbb R$ be a function satisfying
\begin{enumerate}[label=$(\roman*)$,leftmargin=* ,parsep=2pt,itemsep=0pt,topsep=2pt]
	\item $F(0) = 0$ and $F$ is non-decreasing with $\lim_{+\infty}F = +\infty$,
	\item $F$ is lower-semicontinuous and $s\in\mathbb R\mapsto F(e^s)$ is convex.
\end{enumerate}
Then, the sequence $\mathcal M^F$ associated with the function $F$ defined in \eqref{ex_qa_sequence} is quasi-analytic if and only if
$$\int_0^{+\infty} \frac{F(t)}{1+t^2}\ \mathrm dt = +\infty.$$
\end{prop}
The following proposition provides some examples of functions $F$ generating quasi-analytic sequences $\mathcal M^F$. For the sake of completeness of the present work, we give a proof based only on Theorem \ref{Den_Carl_thm}, and not on Proposition \ref{06112020P1}.

\begin{prop}\label{19122020P1} Let $p\geq1$ be a positive integer and $F_p : [0,+\infty)\rightarrow[0,+\infty)$ be the non-negative function defined for all $t\geq0$ by 
$$F_p(t) = \frac t{g(t)(g\circ g)(t)... g^{\circ p}(t)},\quad\text{where}\quad g(t) = \log(e+t),$$
with $g^{\circ p} = g\circ\ldots\circ g$ ($p$ compositions). The associated sequence $\mathcal M^{F_p}$ defined in \eqref{ex_qa_sequence} is a quasi-analytic sequence of positive real numbers.
\end{prop}

\begin{proof} We first notice that the sequence $\mathcal M^{F_p}$ is well-defined by \eqref{ex_qa_sequence} as the supremum of a continuous function on $[0,+\infty)$ with finite limit when $t \to +\infty$. Define $\phi_p : [0,+\infty) \rightarrow [0,+\infty)$ by 
$$\phi_p(t) = g(t)(g\circ g)(t)\ldots g^{\circ p}(t),\quad t\geq0.$$
Let us check that the following convergence holds
\begin{equation}\label{equivalent}
	\frac{\phi_p(k)}{\phi_p\big(2k\phi_p(k)\big)}\underset{k\rightarrow+\infty}{\rightarrow}1.
\end{equation}
Indeed, on the one hand, we get that
$$\phi_p(k)\underset{k\rightarrow+\infty}{\sim}(\log k)\log(\log k)\ldots\log^{\circ p}k.$$
On the other hand, we recall that if $f$ and $g$ are two numerical functions satisfying $f\sim_{+\infty}g$ with $\lim_{+\infty}f = +\infty$, then $\log f\sim_{+\infty}\log g$. As a consequence, we deduce that
$$g(2k\phi_p(k))\underset{k\rightarrow+\infty}{\sim}\log(2k(\log k)\log(\log k)\ldots\log^{\circ p}k)\underset{k\rightarrow+\infty}{\sim}\log k.$$
Iterating this argument, it follows that for all $j\geq1$,
$$g^{\circ j}(2k\phi_p(k))\underset{k\rightarrow+\infty}{\sim}\log^{\circ j}k.$$
This implies that the convergence \eqref{equivalent} actually holds, since
$$\frac{\phi_p(k)}{\phi_p\big(2k\phi_p(k)\big)}\underset{k\rightarrow+\infty}{\sim}\frac{(\log k)\log(\log k)\ldots\log^{\circ p}k}{(\log k)\log(\log k)\ldots\log^{\circ p}k} = 1.$$
Moreover, by direct computations, we have that for all $t\geq0$,
\begin{align}\label{deriv_estimate0}
	\frac{\phi_p'(t)}{\phi_p(t)} & = \sum_{i=1}^p \frac{(g^{\circ i})'(t)}{g^{\circ i}(t)} = \sum_{i=1}^p\frac1{g^{\circ i}(t)}\prod_{j=1}^i g'(g^{\circ(j-1)}(t)) \\[5pt] \nonumber
	& = \sum_{i=1}^p\frac1{g^{\circ i}(t)}\prod_{j=1}^i\frac1{e+g^{\circ(j-1)}(t)} \\[5pt] \nonumber
	& = \frac1{e+t}\sum_{i=1}^p\frac1{g^{\circ i}(t)}\prod_{j=2}^i\frac1{e+g^{\circ(j-1)}(t)}  \\[5pt] \nonumber
	& \underset{t\rightarrow+\infty} = o\bigg(\frac1t\bigg),
\end{align}
that is
\begin{equation}\label{deriv_estimate}
	\phi_p'(t) \underset{t\rightarrow+\infty}= o\Big(\frac{\phi_p(t)}t\Big).
\end{equation}
For all positive integer $k\geq1$, the supremum $M^{F_p}_k =  \sup_{r\geq0}r^ke^{-F_p(r)}$ is reached at a point $t_k>0$ satisfying the following equation
\begin{equation}\label{deriv_estimate2}
	t_kF'_p(t_k)= k,
\end{equation}
which is equivalent to
$$\frac{t_k}{\phi_p(t_k)} \bigg(1-t_k\frac{\phi_p'(t_k)}{\phi_p(t_k)}\bigg)=k.$$
Let us prove that $t_k$ exists and is unique for $k\gg 1$ sufficiently large. First, notice that the application $t \mapsto tF'_p(t)$ satisfies $\lim_{t \to +\infty} tF'_p(t) =+\infty$ since $\phi_p(t)=_{t \to +\infty} o(t)$ and that \eqref{deriv_estimate} holds. It remains to check that there exists $T >0$ such that this application is increasing on $[T, +\infty)$. To that end, it is sufficient to prove that the functions
$$F_p : t \rightarrow \frac{t}{\phi_p(t)}\quad\text{and}\quad G_p : t \rightarrow 1- t \frac{\phi'_p(t)}{\phi_p(t)},$$
are positive and increasing on $[T, +\infty)$ for some $T \gg 1$ sufficiently large. The positivity of these functions is ensured by \eqref{deriv_estimate} and the positivity of $\phi_p$. Let us prove that $F_p$ and $G_p$ are increasing on $[T,+\infty)$ for $T$ sufficiently large. On the one hand, this fact holds true for the first function since 
$$F'_p(t) =\frac1{\phi_p(t)} \bigg(1-t \frac{\phi'_p(t)}{\phi_p(t)}\bigg)\underset{t\rightarrow+\infty}{\sim}\frac1{\phi_p(t)} >0.$$ 
On the other hand, it follows from \eqref{deriv_estimate0} that for all $t \geq 0$,
\begin{align}\label{derivG}
-G'_p(t) &= \bigg(\frac t{e+t}\sum_{i=1}^p\frac1{g^{\circ i}(t)}\prod_{j=2}^i\frac1{e+g^{\circ(j-1)}(t)} \bigg)' \\[5pt] \nonumber
&= \frac e{(e+t)^2} \sum_{i=1}^p\frac1{g^{\circ i}(t)}\prod_{j=2}^i\frac1{e+g^{\circ(j-1)}(t)} + \frac{t}{e+t} \bigg(\sum_{i=1}^p\frac1{g^{\circ i}(t)}\prod_{j=2}^i\frac1{e+g^{\circ(j-1)}(t)} \bigg)'.
\end{align}
It can be checked that there exists a positive function $h_2$ given by a rational fraction of $g, g^{\circ2},\ldots, g^{\circ p}$ such that
\begin{equation}\label{rationalfrac}
	\bigg(\sum_{i=1}^p\frac1{g^{\circ i}(t)}\prod_{j=2}^i\frac1{e+g^{\circ(j-1)}(t)}\bigg)' = -\frac{h_2(t)}{e+t}.
\end{equation}
We deduce from \eqref{derivG} and \eqref{rationalfrac} that 
\begin{equation*}
	-G'_p(t)  = \frac e{(e+t)^2} h_1(t)- \frac t{(e+t)^2} h_2(t) = \frac t{(e+t)^2} h_2(t) \bigg( \frac et \frac{h_1(t)}{h_2(t)} -1 \bigg),
\end{equation*}
with
\begin{equation*}
\forall t \geq 0, \quad h_1(t) =\sum_{i=1}^p\frac1{g^{\circ i}(t)}\prod_{j=2}^i\frac1{e+g^{\circ(j-1)}(t)}.
\end{equation*}
Since $h_1/h_2$ is a rational fraction of $g, g^{\circ2},\ldots,g^{\circ p}$, classical lemmas ensure that 
$$\frac{h_1(t)}{t h_2(t)} \underset{t \to +\infty}{\rightarrow}0.$$ 
Therefore, since $h_2$ is positive, there exists a positive constant $T>0$ such that $G'_p(t) >0$ for all $t\geq T$. It follows that for all $k\gg1$ sufficiently large, there exists a unique $t_k \in(0,+\infty)$ satisfying \eqref{deriv_estimate2}. By defining $u_k = 2k\phi_p(k)$, it follows from \eqref{equivalent} and \eqref{deriv_estimate} that
$$u_kF'_p(u_k) = \frac{u_k}{\phi_p(u_k)} \bigg(1-u_k \frac{\phi_p'(u_k)}{\phi_p(u_k)}\bigg)= 2k \frac{\phi_p(k)}{\phi_p\big(2k\phi_p(k)\big)} \bigg(1-u_k \frac{\phi_p'(u_k)}{\phi_p(u_k)} \bigg)\underset{k\rightarrow+\infty}{\sim} 2k.$$
As a consequence, since $t_kF'_p(t_k) = k$ from \eqref{deriv_estimate2} and that the application $t \mapsto tF'_p(t)$ is increasing on $[T,+\infty)$, we get that for all $k\gg1$ sufficiently large,
$$t_k \le u_k = 2k\phi_p(k).$$
We deduce that for all $k\gg1$,
$$\frac{M^{F_p}_{k-1}}{M^{F_p}_k}\geq\frac{t_k^{k-1}e^{-F_p(t_k)}}{t_k^ke^{-F_p(t_k)}}=\frac1{t_k}\geq\frac1{2k\phi_p(k)}.$$
It is well-known that the series $\sum\frac1{2k\phi_p(k)}$ is divergent and then, we obtain 
$$\sum_{k=0}^{+\infty}\frac{M^{F_p}_k}{M^{F_p}_{k+1}} =+\infty.$$
Since the sequence $\mathcal M^{F_p}$ is log-convex by construction, we deduce from Theorem \ref{Den_Carl_thm} that the sequence $\mathcal M^{F_p}$ is quasi-analytic. This ends the proof of Proposition \ref{19122020P1}.
\end{proof}

\begin{cor}\label{21122020C1} Let $s\geq1$, $0\le\delta\le1$ be non-negative real numbers and $F_{s,\delta} : [0,+\infty)\rightarrow[0,+\infty)$ be the non-negative continuous function defined by
$$F_{s, \delta}(t)= \frac{t^s}{\log^{\delta}(e+t)},\quad t\geq0.$$
The associated sequence $\mathcal M^{F_{s,\delta}}$ defined in \eqref{ex_qa_sequence} is a quasi-analytic sequence of positive real numbers.
\end{cor}

\begin{proof} Notice that there exists a positive constant $C_{s,\delta}>0$ such that for all $t \geq 0$,
$$\frac{t^s}{\log^{\delta}(e+t)} \geq \frac t{\log(e+t)} - C_{s,\delta}.$$
Corollary \ref{21122020C1} is then a consequence of Lemma \ref{17122020L1}, Lemma \ref{17122020L2} and Proposition \ref{19122020P1}.
\end{proof}

\section{Proof of the approximate null-controllability results}\label{cunc}

This section is devoted to the proofs of the results stated in Section~\ref{approximate_cont_result}, dealing with the (cost-uniform) approximate null-controllability of the evolution equation
\begin{equation}\label{diff_eq0}\tag{$E_F$}
\left\{\begin{aligned}
	& \partial_tf(t,x) + F(\vert D_x\vert) f(t,x) = h(t,x)\mathbbm 1_{\omega}(x),\quad t>0,\ x\in\mathbb R^n, \\
	& f(0,\cdot) = f_0\in L^2(\mathbb R^n),
\end{aligned}\right.
\end{equation}
where the continuous function $F : [0,+\infty) \rightarrow \mathbb R$ bounded from below is assumed to be associated with a quasi-analytic sequence $\mathcal{M}^F$ of positive real numbers, defined in \eqref{16122020E1}, and $\omega\subset\mathbb R^n$ is a measurable set with positive Lebesgue measure.

\subsection{Proof of Proposition \ref{approximate_controllability_weak}}\label{approx_cont_proof} The purpose of this first subsection is to prove that the evolution equation \eqref{diff_eq0} is approximately null-controllable from the support $\omega$ in any positive time $T>0$ (with no extra assumption on $\omega$). It is well-known, see e.g. \cite{coron_book} (Theorem 2.43, page 56) or \cite{MR4041279} (Remark 16), that the approximate null-controllability (without uniform cost) of the system \eqref{01062020E1} is equivalent to a unique continuation property of the adjoint system. More precisely, the adjoint system 
$$\left\{\begin{aligned}
	& \partial_tg(t,x) + F(\vert D_x\vert) g(t,x) = 0,\quad t>0,\ x\in\mathbb R^n, \\
	& g(0,\cdot) = g_0\in L^2(\mathbb R^n),
\end{aligned}\right.$$
is said to satisfy the unique continuation property from $\omega\subset\mathbb R^n$ at some time $T>0$ if for all initial datum $g_0\in L^2(\mathbb R^n)$, we have
$$\big(\forall t\in[0,T],\ \mathbbm1_{\omega}e^{-tF(\vert D_x\vert)}g_0=0\big)\Rightarrow\big(e^{-TF(\vert D_x\vert)}g_0=0\big).$$
Let $T>0$ be a fixed positive time. We aim at proving that the unique continuation property holds from the control support $\omega$.

Let us first establish that for all $g \in L^2(\mathbb{R}^n)$, the function $e^{-TF(\vert D_x\vert)}g$ belongs to a class of quasi-analytic functions. We consider $g\in L^2(\mathbb R^n)$ not being identically equal to zero (the result is straightforward when $g=0$). Notice that for all $\beta \in \mathbb{N}^n$, 
$$\xi \rightarrow \xi^{\beta} \mathscr F\big( e^{-TF(\vert D_x\vert)}g\big) \in L^1(\mathbb{R}^n),$$ 
since we have by Plancherel's theorem that
\begin{align}\label{sobolev_embedding0}	
\int_{\mathbb R^n} \big\vert\xi^{\beta} \mathscr F\big( e^{-TF(\vert D_x\vert)}g\big)(\xi)\big\vert\ \mathrm d\xi & = \int_{\mathbb R^n} \big\vert\xi^{\beta} e^{-TF(|\xi|)} \widehat g(\xi)\big\vert\ \mathrm d\xi \\[5pt]
	& \leq (2\pi)^{n/2} \bigg(\int_{\mathbb{R}^n} |\xi|^{2|\beta|} e^{-2TF(|\xi|)}\ \mathrm d\xi \bigg)^{\frac12} \|g\|_{L^2(\mathbb R^n)} \nonumber \\[5pt]
	& = (2\pi)^{n/2}\bigg(\int_{\mathbb{R}^n} \frac{(|\xi|^{2|\beta|}+|\xi|^{2|\beta|+2n}) e^{-2TF(|\xi|)}}{1+|\xi|^{2n}}\ \mathrm d\xi \bigg)^{\frac12} \|g\|_{L^2(\mathbb{R}^n)} \nonumber \\[5pt]
	& \leq C_n \big((M^{TF}_{|\beta|})^2 + (M^{TF}_{|\beta|+n})^2\big)^{\frac12} \|g\|_{L^2(\mathbb{R}^n)} \nonumber \\[5pt]
	& < +\infty, \nonumber
\end{align}
the positive constant $C_n>0$ being given by
$$C_n=(2\pi)^{n/2} \bigg(\int_{\mathbb R^n} \frac1{1+|\xi|^{2n}}\ \mathrm d\xi \bigg)^{\frac12}<+\infty.$$
Since the sequence $\mathcal M^{TF}$ is log-convex, there exists a positive constant $A_{TF}>0$ such that 
$$\forall k \geq 0, \quad M^{TF}_k \leq A_{TF} M^{TF}_{k+1},$$
and we get from \eqref{sobolev_embedding0} that for all $\beta \in \mathbb N^n$,
$$\int_{\mathbb R^n} \big\vert\xi^{\beta} \mathscr F\big( e^{-TF(\vert D_x\vert)}g\big)(\xi)\big\vert\ \mathrm d\xi \leq C_{n, TF} M^{TF}_{|\beta|+n} \|g\|_{L^2(\mathbb{R}^n)},$$
with $C_{n, TF} = C_n(1+ A^{2n}_{TF})^{\frac12}$. It then follows from the Fourier inverse formula that
$$\forall\beta\in\mathbb N^n,\quad\partial^{\beta}_x\big(e^{-TF(\vert D_x\vert)}g\big) \in L^{\infty}(\mathbb{R}^n),$$ 
and that there exists a positive constant $C'_{n,TF}>0$ such that
\begin{equation}\label{QA_estimate0}
	\forall\beta\in\mathbb N^n,\quad \big\Vert\partial^{\beta}_x \big(e^{-TF(\vert D_x\vert)}g\big)\big\Vert_{L^{\infty}(\mathbb R^n)} \leq C'_{n,TF} M^{TF}_{|\beta|+n} \|g\|_{L^2(\mathbb R^n)}.
\end{equation}
Notice that for all positive real number $\lambda>0$, the sequence $\mathcal M_{\lambda}$ whose elements $M_{\lambda,k}$ are given for all $k\geq0$ by 
\begin{equation}\label{23122020E1}
	M_{\lambda,k} = C'_{n, TF}\lambda^kM^{TF}_{k+n}\Vert g\Vert_{L^2(\mathbb R^n)},
\end{equation}
also defines a log-convex sequence which satisfies
$$\sum_{k = 0}^{+\infty} \frac{M_{\lambda,k}}{M_{\lambda,k+1}} = \frac1{\lambda}\sum_{k = n}^{+\infty} \frac{M^{TF}_k}{M^{TF}_{k+1}}.$$
Since the sequence $\mathcal M^F$ is quasi-analytic, so is the sequence $\mathcal M^{TF}$ according to Lemma~\ref{17122020L2} and therefore, Theorem \ref{Den_Carl_thm} implies that
$$\sum_{k = 0}^{+\infty} \frac{M^{TF}_k}{M^{TF}_{k+1}}=+\infty.$$
A second application of this theorem shows that the log-convex sequence $\mathcal M_{\lambda}$ is also quasi-analytic. This fact, combined with the estimate \eqref{QA_estimate0} and Proposition \ref{23122020P1}, implies that for all $g\in L^2(\mathbb R^n)$, the function $e^{-TF(\vert D_x\vert)}g$ belongs to the quasi-analytic class $\mathcal C_{\mathcal M_1}(\mathbb R^n)$.

Now, let us check that the unique continuation property holds. To that end, we make a weaker assumption by considering a function $g \in L^2(\mathbb R^n)$ not identically equal to zero such that 
\begin{equation}\label{06012020E1}
	\mathbbm1_{\omega}e^{-TF(\vert D_x\vert)}g=0.
\end{equation}
The purpose is to prove that $e^{-TF(|D_x|)}g=0$. Since $\Leb(\omega)>0$, if we consider the cube $Q=(-1,1)^n$, then
$$\exists\lambda_0>0,\forall\lambda\geq\lambda_0,\quad \Leb(\omega \cap \lambda Q)>0.$$
Let us consider $\lambda \geq \lambda_0$ and define the function $g_{\lambda}\in L^2(\mathbb R^n)$ by
$$g_{\lambda}(x) = \big(e^{-TF(|D_x|)} g\big) (\lambda x),\quad x\in\mathbb R^n.$$
It readily follows from the estimate \eqref{QA_estimate0} that for all $\beta \in \mathbb N^n$,
$$\big\Vert\partial^{\beta}_x g_{\lambda}\big\Vert_{L^{\infty}(Q)} \leq \lambda^{|\beta|} \big\Vert\partial^{\beta}_x \big(e^{-TF(\vert D_x\vert)}g\big)\big\Vert_{L^{\infty}(\mathbb{R}^n)} \leq C'_{n, TF}\lambda^{|\beta|} M^{TF}_{|\beta|+n} \|g\|_{L^2(\mathbb R^n)}.$$
Since the sequences $\mathcal M_{\lambda}$ defined in \eqref{23122020E1} are quasi-analytic and that $g_{\lambda}\in\mathcal C_{\mathcal M_{\lambda}}(\mathbb R^n)$ from the above estimate, Proposition~\ref{15122020P1} implies that for all $\lambda \geq \lambda_0$, there exists a positive constant $C = C(\lambda,\Vert g_{\lambda}\Vert_{L^{\infty}(Q)},\omega,n)>0$
such that
$$\|g_{\lambda}\|_{L^2(Q)}  \leq C \|g_{\lambda} \|_{L^2((\lambda^{-1}\omega) \cap Q)}.$$
As a consequence, we get from the assumption \eqref{06012020E1} that for all $\lambda\geq\lambda_0$,
$$\big\Vert e^{-TF(|D_x|)} g\big\Vert_{L^2(\lambda Q)} \leq C\big\Vert e^{-TF(|D_x|)}g\big\Vert_{L^2(\omega \cap \lambda Q)} =0.$$
It follows that $\mathbbm{1}_{\lambda Q} e^{-TF(|D_x|)}g =0$ for all $\lambda \geq \lambda_0$, and therefore $e^{-TF(|D_x|)}g=0$. Then, the unique continuation property holds and the evolution equation \eqref{01062020E1} is approximately null-controllable from $\omega$ at time $T>0$.

\subsection{Proof of Theorem \ref{21102020T1}: sufficient part} In this subsection, we aim at proving that when the support control $\omega$ is thick, then the equation \eqref{diff_eq0} is cost-uniformly approximately null-controllable from $\omega$ in any positive time $T>0$. To that end, we shall use the following observability characterization of the cost-uniform approximate null-controllability:

\begin{thm}[Proposition 6 in \cite{MR4041279}]\label{21102020T2} Let $T>0$ be a positive time. The following assertions are equivalent:
\begin{enumerate}[label=$(\roman*)$,leftmargin=* ,parsep=2pt,itemsep=0pt,topsep=2pt]
	\item The evolution system \eqref{01062020E1} is cost-uniformly approximately null-controllable from $\omega$ in time $T>0$.
	\item For all $\varepsilon\in(0,1)$, there exists a positive constant $C_{\varepsilon,T}>0$ such that for all $g\in L^2(\mathbb R^n)$,
	\begin{equation}\label{06112020E6}
	\big\Vert e^{-TF(\vert D_x\vert)}g\big\Vert^2_{L^2(\mathbb R^n)}\le C_{\varepsilon, T}\int_0^T\big\Vert e^{-tF(\vert D_x\vert)}g\big\Vert^2_{L^2(\omega)}\ \mathrm dt 
	+ \varepsilon\Vert g\Vert^2_{L^2(\mathbb R^n)}.
	\end{equation}
\end{enumerate}
\end{thm}

According to Theorem \ref{21102020T2}, we need to prove that the observability estimate \eqref{06112020E6} holds. The first step consists in establishing the following uncertainty principle: 

\begin{prop}\label{uncertainty_principle} For all time $T>0$ and for all $\varepsilon>0$, there exists a positive constant $C_{\varepsilon, T} >0$ such that for all $t\geq T/2$ and $g\in L^2(\mathbb R^n)$,
$$\big\Vert e^{-t F(\vert D_x\vert)}g\big\Vert^2_{L^2(\mathbb R^n)} \le C_{\varepsilon, T} \big\|e^{-t F(\vert D_x\vert)} g \big\|^2_{L^2(\omega)} + \varepsilon e^{-2t\inf F}\Vert g\Vert^2_{L^2(\mathbb R^n)}.$$
\end{prop}

\begin{proof} This proof is adapted from the one of \cite{JM} (Theorem 1.3). Let $T>0$ and $\varepsilon>0$ fixed. Considering the function $G = F-\inf F$, we begin by proving that there exists a positive constant $C_{\varepsilon, T} >0$ such that for all $t\geq T/2$ and $g\in L^2(\mathbb R^n)$,
\begin{equation}\label{12042021E1}
	\big\Vert e^{-tG(\vert D_x\vert)}g\big\Vert^2_{L^2(\mathbb R^n)} \le C_{\varepsilon, T} \big\Vert e^{-tG(\vert D_x\vert)}g\big\Vert^2_{L^2(\omega)} + \varepsilon\Vert g\Vert^2_{L^2(\mathbb R^n)}.
\end{equation}
Once the estimate \eqref{12042021E1} is established, we deduce that for all $t\geq T/2$ and $g\in L^2(\mathbb R^n)$,
$$\big\Vert e^{-tF(\vert D_x\vert)}g\big\Vert^2_{L^2(\mathbb R^n)} \le C_{\varepsilon, T}\big\Vert e^{-tF(\vert D_x\vert)}g\big\Vert^2_{L^2(\omega)} + \varepsilon e^{-2t\inf F}\Vert g\Vert^2_{L^2(\mathbb R^n)},$$
and the proof of Proposition \ref{uncertainty_principle} is therefore ended.

Let us prove the estimate \eqref{12042021E1}. Setting $T'=T/2$, we first check, using Parseval's formula, that the following Bernstein type estimates hold for all $g\in L^2(\mathbb R^n)$, $t\geq T'$ and $\beta\in\mathbb N^n$:
\begin{align}\label{bernstein_estimate}
	\big\Vert\partial^{\beta}_x(e^{-tG(\vert D_x\vert)}g)\big\Vert^2_{L^2(\mathbb R^n)} 
	& = \frac1{(2\pi)^n}\int_{\mathbb R^n}e^{-2tG(\vert\xi\vert)}\vert\xi^{\beta}\vert^2\vert\widehat g(\xi)\vert^2\ \mathrm d\xi \\[5pt]
	& \le\frac1{(2\pi)^n}\int_{\mathbb R^n}e^{-2T'G(\vert\xi\vert)}\vert\xi\vert^{2\vert\beta\vert}\vert\widehat g(\xi)\vert^2\ \mathrm d\xi \nonumber \\[5pt]
	& \le\Big(\sup_{r\geq0}r^{\vert\beta\vert}e^{-T'G(r)}\Big)^2\Vert g\Vert^2_{L^2(\mathbb R^n)} \nonumber \\[5pt]
	& = (M^{T'G}_{\vert\beta\vert})^2\Vert g\Vert^2_{L^2(\mathbb R^n)}. \nonumber
\end{align}
These estimates will be central later to control the functions $e^{-tG(\vert D_x\vert)}g$ on bad cubes, which we are about to define. Since $\omega$ is a thick set, by definition, there exist $\gamma\in(0,1]$ and a length $L>0$ such that
$$\forall x\in\mathbb R^n, \quad\Leb(\omega\cap(x + [0,L]^n))\geq\gamma L^n.$$
For $\alpha\in L\mathbb Z^n$, let us define $Q(\alpha)$ as the following cube
$$Q(\alpha) = \alpha + [0,L]^n.$$
Notice that the family of cubes $(Q(\alpha))_{\alpha\in L\mathbb Z^n}$ covers the space $\mathbb R^n$:
\begin{equation}\label{06112020E5}
	\mathbb R^n = \bigcup_{\alpha\in L\mathbb Z^n}Q(\alpha).
\end{equation}
Let $g \in L^2(\mathbb R^n)$ and $t \geq T'$, where we recall that $T'=T/2$. A cube $Q(\alpha)$ is said to be good if it satisfies the property
\begin{equation}\label{08012020E1}
	\forall\beta\in\mathbb N^n, \quad\big\Vert\partial^{\beta}_x(e^{-tG(\vert D_x\vert)}g)\big\Vert^2_{L^2(Q(\alpha))}\le\frac{2^{2\vert\beta\vert+n}}{\varepsilon}(M^{T'G}_{\vert\beta\vert})^2\big\Vert e^{-tG(\vert D_x\vert)}g\big\Vert^2_{L^2(Q(\alpha))},
\end{equation}
where the positive real numbers $M^{T'G}_k$ are the elements of the sequence $\mathcal M^{T'G}$ generated by the function $T'G$. We recall that for all $k\geq0$, 
$$M^{T'G}_k=\sup_{r\geq0}r^ke^{-T'G(r)}.$$
Naturally, a cube $Q(\alpha)$ is said to be bad if it is not good, that is, when
\begin{equation}\label{bad_cube}
	\exists\beta_0\in\mathbb N^n,\quad\big\Vert\partial^{\beta_0}_x(e^{-tG(\vert D_x\vert)}g)\big\Vert^2_{L^2(Q(\alpha))}>\frac{2^{2\vert\beta_0\vert+n}}{\varepsilon}(M^{T'G}_{\vert\beta_0\vert})^2\big\Vert e^{-tG(\vert D_x\vert)}g\big\Vert^2_{L^2(Q(\alpha))}.
\end{equation}
Notice from the covering property \eqref{06112020E5} that
\begin{equation}\label{gb1}
	\big\Vert e^{-tG(\vert D_x\vert)}g\big\Vert^2_{L^2(\mathbb R^n)} = \sum_{\text{good cubes}} \big\Vert e^{-tG(\vert D_x\vert)}g\big\Vert^2_{L^2(Q(\alpha))} +\sum_{\text{bad cubes}}\big\Vert e^{-tG(\vert D_x\vert)}g\big\Vert^2_{L^2(Q(\alpha))}.
\end{equation}
We shall estimate independently the two terms of the right-hand side of this equality. Let us begin with the second one. 
It follows from the definition \eqref{bad_cube} that if $Q(\alpha)$ is a bad cube, there exists $\beta_0\in\mathbb N^n$ such that
\begin{align}\label{bad_cube2}
	\big\Vert e^{-tG(\vert D_x\vert)}g\big\Vert^2_{L^2(Q(\alpha))} & \le\frac{\varepsilon}{2^{2\vert\beta_0\vert+n}(M^{T'G}_{\vert\beta_0\vert})^2}\big\Vert\partial^{\beta_0}_x(e^{-tG(\vert D_x\vert)}g)\big\Vert^2_{L^2(Q(\alpha))} \\[5pt]
	& \le\sum_{\beta\in\mathbb N^n}\frac{\varepsilon}{2^{2\vert\beta\vert+n}(M^{T'G}_{\vert\beta\vert})^2}\big\Vert\partial^{\beta}_x(e^{-tG(\vert D_x\vert)}g)\big\Vert^2_{L^2(Q(\alpha))}. \nonumber
\end{align}
By summing over all the bad cubes, we obtain from \eqref{bernstein_estimate} and \eqref{bad_cube2} that 
\begin{align}\label{12042021E2}
	\int_{\bigcup_{\text{bad cubes}}} \big\vert e^{-tG(\vert D_x\vert)}g(x)\big\vert^2\ \mathrm dx
	& \le\varepsilon\sum_{\text{bad cubes}}\sum_{\beta\in\mathbb N^n}\frac1{2^{2\vert\beta\vert+n}(M^{T'G}_{\vert\beta\vert})^2}\big\Vert\partial^{\beta}_x(e^{-tG(\vert D_x\vert)}g)\big\Vert^2_{L^2(Q(\alpha))} \\[5pt]
	& \le\varepsilon\sum_{\beta\in\mathbb N^n}\frac 1{2^{2\vert\beta\vert+n}(M^{T'G}_{\vert\beta\vert})^2}\big\Vert\partial^{\beta}_x(e^{-tG(\vert D_x\vert)}g)\big\Vert^2_{L^2(\mathbb R^n)} \nonumber \\[5pt]
	& \leq\varepsilon\sum_{\beta\in\mathbb N^n}\frac1{2^{2\vert\beta\vert+n}}\Vert g\Vert^2_{L^2(\mathbb R^n)} \nonumber \\[5pt]
	& \le \varepsilon\Vert g \Vert^2_{L^2(\mathbb R^n)}, \nonumber
\end{align}
since
$$\sum_{\beta\in\mathbb N^n}\frac1{2^{2\vert\beta\vert+n}}=\sum_{k=0}^{+\infty}\binom{k+n-1}k\frac1{2^{2k+n}}\le\sum_{k=0}^{+\infty}2^{k+n-1}\frac1{2^{2k+n}} =1.$$
It remains to estimate the first term of the right-hand side of the equality \eqref{gb1}. Let $Q(\alpha)$ be a good cube. Assuming that the function $e^{-tG(\vert D_x\vert)}g$ is not identically equal to zero on the cube $Q(\alpha)$, we define the following function $\varphi:[0,1]^n\rightarrow\mathbb C$ by
\begin{equation}\label{function_aux}
	\varphi(z)= \frac{L^{\frac n2}\big(e^{-tG(\vert D_x\vert)}g\big)(\alpha + Lz)}{\Vert e^{-tG(\vert D_x\vert)}g\Vert_{L^2(Q(\alpha))}},\quad z\in[0,1]^n.
\end{equation} 
As the cube $[0,1]^n$ satisfies the cone condition, the following Sobolev embedding holds
$$W^{n,2}([0,1]^n) \hookrightarrow L^{\infty}([0,1]^n),$$
see e.g.~\cite{adams} (Theorem~4.12). This implies that there exists a positive constant $C_n>0$, only depending on the dimension $n \geq 1$, such that 
\begin{equation}\label{sobolev}
	\forall u \in W^{n,2}([0,1]^n), \quad \Vert u\Vert_{L^{\infty}([0,1]^n)}\le C_n\Vert u\Vert_{W^{n,2}([0,1]^n)}.
\end{equation}
From the definition \eqref{08012020E1} of good cube and the estimate \eqref{sobolev}, it follows that for all $\beta\in\mathbb N^n$,
\begin{align}
	\big\Vert\partial^{\beta}_x\varphi\big\Vert^2_{L^{\infty}([0,1]^n)}
	& \le C_n^2\sum_{\tilde{\beta}\in\mathbb N^n,\ \vert\tilde{\beta}\vert\le n}\big\Vert\partial^{\beta+\tilde{\beta}}_x\varphi\big\Vert^2_{L^2([0,1]^n)} \\
	 & = \frac{C_n^2}{\Vert e^{-tG(\vert D_x\vert)}g\Vert_{L^2(Q(\alpha))}^2}L^{2\vert\beta\vert}\sum_{\tilde{\beta}\in\mathbb N^n,\ \vert\tilde{\beta}\vert\le n}L^{2\vert\tilde{\beta}\vert}\big\Vert\partial^{\beta+\tilde{\beta}}_xe^{-tG(\vert D_x\vert)}g\big\Vert^2_{L^2(Q(\alpha))} \nonumber \\[5pt]
	& \le C_n^2 (2L)^{2\vert\beta\vert}2^n\varepsilon^{-1}\sum_{\tilde{\beta}\in\mathbb N^n,\ \vert\tilde{\beta}\vert\le n}(2L)^{2\vert\tilde{\beta}\vert}(M^{T'G}_{\vert\beta\vert+\vert\tilde{\beta}\vert})^2. \nonumber
\end{align}
Since the sequence $\mathcal M^{T'G}$ is log-convex, there exists a positive constant $A_{T'G}>0$ such that 
$$\forall k\geq0, \quad M^{T'G}_k \le A_{T'G} M^{T'G}_{k+1}.$$
Therefore, there exists a new positive constant $D_{n,L} >0$ depending on $n$ and $L$, such that for all $\beta\in\mathbb N^n$,
\begin{equation}\label{21122020E2}
	\big\Vert\partial^{\beta}_x\varphi\big\Vert^2_{L^{\infty}([0,1]^n)}\le\varepsilon^{-1} A^{2n}_{T'G} D_{n,L}^2(2L)^{2\vert\beta\vert}(M^{T'G}_{\vert\beta\vert+n})^2.
\end{equation}
Let us recall that by assumption, the sequence $\mathcal M^F$ associated with the function $F$ defines a quasi-analytic sequence of positive real numbers, and so is $\mathcal M^{T'G}$ according to Lemma \ref{17122020L2}. Notice that the slightly modified sequence $\mathcal M$, whose elements $M_k$ are given for all $k\geq0$ by 
$$M_k = \varepsilon^{-{1/2}}A^n_{T'G} D_{n,L}(2L)^kM^{T'G}_{k+n},$$ 
also defines a log-convex sequence that satisfies
$$\sum_{k = 0}^{+\infty} \frac{M_k}{M_{k+1}} = \frac1{2L}\sum_{k=n}^{+\infty}\frac{M^{T'G}_k}{M^{T'G}_{k+1}}.$$
Since the sequence $\mathcal M^{T'G}$ is log-convex and quasi-analytic, Theorem \ref{Den_Carl_thm} implies that
$$\sum_{k = 0}^{+\infty} \frac{M^{T'G}_k}{M^{T'G}_{k+1}}=+\infty,$$
and a second application of this theorem provides that the log-convex sequence $\mathcal M$ is also quasi-analytic. Since $\varphi\in\mathcal C_{\mathcal M}((0,1)^n)$, and that this function also satisfies 
$$\Vert \varphi\Vert_{L^{\infty}([0,1]^n)} = \frac{L^{\frac n2} \Vert e^{-tG(\vert D_x\vert)}g\Vert_{L^{\infty}(Q(\alpha))}}{\Vert e^{-tG(\vert D_x\vert)}g\Vert_{L^2(Q(\alpha))}}\geq1,$$
we therefore deduce from \eqref{21122020E2} and Proposition \ref{15122020P1} that there exists a positive constant $C_{\varepsilon,\gamma,L,n,T,G}>0$ independent on $g$ and $\alpha$ such that 
\begin{equation}\label{quasi_analytic0}
	\Vert\varphi\Vert_{L^2([0,1]^n)}\le C_{\varepsilon,\gamma,L,n,T,G}\Vert\varphi\Vert_{L^2(E)},
\end{equation}
where $E= \frac{\omega - \alpha}L \cap [0,1]^n \subset [0,1]^n$ satisfies $\Leb(E)\geq\gamma >0$. It follows directly from \eqref{function_aux} and \eqref{quasi_analytic0} that
\begin{equation}\label{quasi_analytic1}
	\big\Vert e^{-tG(\vert D_x\vert)}g\big\Vert_{L^2(Q(\alpha))}\le C_{\varepsilon,\gamma,L,n,T,G}\big\Vert e^{-tG(\vert D_x\vert)}g\big\Vert_{L^2(\omega\cap Q(\alpha))}.
\end{equation}
Clearly, the above estimate also holds when the function $e^{-tG(\vert D_x\vert)}g$ is identically equal to zero on the cube $Q(\alpha)$. By summing over all the good cubes, we deduce from \eqref{06112020E5} and \eqref{quasi_analytic1} that
$$\int_{\bigcup_{\text{good cubes}}Q(\alpha)}\big\vert e^{-tG(\vert D_x\vert)}g(x)\big\vert^2\ \mathrm dx \leq C^2_{\varepsilon,\gamma,L,n,T,G}\int_{\omega}\big\vert e^{-tG(\vert D_x\vert)}g(x)\big\vert^2\ \mathrm dx.$$
This estimate, together with \eqref{gb1} and \eqref{12042021E2}, imply that the estimate \eqref{12042021E1} actually holds. The proof of Proposition~\ref{uncertainty_principle} is therefore ended.
\end{proof}

We can now tackle the proof of the observability estimate \eqref{06112020E6}. Let $\varepsilon >0$ and $T>0$. It follows from Proposition \ref{uncertainty_principle} that there exists a positive constant $C_{\varepsilon,T}>0$ such that for all $t\geq T/2$ and $g\in L^2(\mathbb R^n)$,
$$\big\Vert e^{-tF(\vert D_x\vert)}g\big\Vert^2_{L^2(\mathbb R^n)} \le C_{\varepsilon, T} \big\Vert e^{-t F(\vert D_x\vert)} g \big\Vert^2_{L^2(\omega)} + \varepsilon'e^{-2t\inf F}\Vert g\Vert^2_{L^2(\mathbb R^n)},$$
where we set
$$\varepsilon' = e^{2T\inf F}\varepsilon.$$
By using that for all $s_1\geq s_2\geq0$ and $g\in L^2(\mathbb R^n)$,
$$\big\Vert e^{-s_1F(\vert D_x\vert)} g\big\Vert_{L^2(\mathbb R^n)}\le e^{(s_2-s_1)\inf F}\big\Vert e^{-s_2F(\vert D_x\vert)} g\big\Vert_{L^2(\mathbb R^n)},$$
it follows that for all $g\in L^2(\mathbb R^n)$,
\begin{align*}
	&\ \big\Vert e^{-TF(\vert D_x\vert)} g\big\Vert^2_{L^2(\mathbb R^n)} \\[5pt]
	\le &\ \frac 2T \int_{\frac T2}^Te^{2(t-T)\inf F}\big\Vert e^{-tF(\vert D_x\vert)}g\big\Vert^2_{L^2(\mathbb R^n)}\ \mathrm dt \\[5pt]
	\le &\ \frac{2C_{\varepsilon,T}}T\int_{\frac T2}^Te^{2(t-T)\inf F}\big\Vert e^{-tF(\vert D_x\vert)}g\big\Vert^2_{L^2(\omega)}\ \mathrm dt + e^{-2T\inf F}\varepsilon'\Vert g\Vert^2_{L^2(\mathbb R^n)} \\[5pt]
	\le &\ \frac{2C_{\varepsilon,T}}Te^{-T(\inf F)_-}\int_0^T\big\Vert e^{-tF(\vert D_x\vert)}g\big\Vert^2_{L^2(\omega)}\ \mathrm dt + \varepsilon\Vert g\Vert^2_{L^2(\mathbb R^n)},
\end{align*}
with $(\inf F)_- = \min(\inf F,0)$. This ends the proof of the estimate \eqref{06112020E6}.



\subsection{Cost-uniform approximate null-controllability vs rapid stabilization}\label{cunc_rapid_stab} To end this section, let us quickly check that cost-uniform approximate null-controllability implies rapid stabilization:

\begin{prop}\label{06012021P1} If the control system \eqref{01062020E1} is cost-uniformly approximately null-controllable from $\omega$ at some positive time $T>0$, then  is rapidly stabilizable from $\omega$.
\end{prop}

\begin{proof} Assume that the system \eqref{01062020E1} is cost-uniformly approximately null-controllable from $\omega$ at some positive time $T>0$. According to Theorem~\ref{21102020T2}, the following observability estimate holds: for all $\varepsilon >0$, there exists a positive constant $C_{\varepsilon,T} >0$ such that for all $g\in L^2(\mathbb R^n)$,
\begin{equation}\label{obs_cunc0}
	\big\Vert e^{-TF(\vert D_x\vert)}g\big\Vert^2_{L^2(\mathbb R^n)}\le C_{\varepsilon, T}\int_0^T\big\Vert e^{-tF(\vert D_x\vert)}g\big\Vert^2_{L^2(\omega)}\ \mathrm dt + \varepsilon\Vert g\Vert^2_{L^2(\mathbb R^n)}.
\end{equation}
In order to prove that \eqref{01062020E1} is rapidly stabilizable, it is sufficient to show that for all $\mu >0$, the system (\hyperref[01062020E1]{$E_{F_{\mu}}$}) is stabilizable, where $F_{\mu} = F- \mu$. To that end, we shall apply Theorem~\ref{29102020T2}, already used in Section~\ref{necessary_part}. From \eqref{obs_cunc0}, by multiplying by $e^{2T\mu}$, we get that for all $\varepsilon >0$, there exists a positive constant $C_{\varepsilon,T} >0$ such that for all $g\in L^2(\mathbb R^n)$,
$$\big\Vert e^{-TF_{\mu}(\vert D_x\vert)}g\big\Vert^2_{L^2(\mathbb R^n)}\le C_{\varepsilon, T}\int_0^T\big\Vert e^{-tF_{\mu}(\vert D_x\vert)}g\big\Vert^2_{L^2(\omega)}\ \mathrm dt + \varepsilon e^{2T\mu}\Vert g\Vert^2_{L^2(\mathbb R^n)}.$$
By taking $\varepsilon = \frac{e^{-2T\mu}}2$ in the above estimate, it follows from the characterization $(iii)$ of Theorem~\ref{29102020T2} that the system (\hyperref[01062020E1]{$E_{F_{\mu}}$}) is stabilizable from $\omega$. Since $\mu>0$ can be chosen arbitrary large, we deduce that the system \eqref{01062020E1} is rapidly stabilizable from $\omega$.
\end{proof}

\section{Appendix}\label{appendix}

This appendix is devoted to the proof of Proposition~\ref{18122020T2}. The following arguments are due to H. Liu, G. Wang, Y. Xu and H. Yu and are originally presented in \cite{LWXY} (Section 2). Let $F:[0,+\infty)\rightarrow\mathbb R$ be a continuous function bounded from below and $\omega\subset\mathbb R^n$ be a Borel set. Assume that the control system \eqref{01062020E1} is stabilizable from $\omega$ at rate $\alpha >0$. We aim at proving that there exists a positive constant $A_{\alpha}>0$ such that for all $T>0$, there exists a positive constant $C_{\alpha, T} >0$ satisfying that for all $g\in L^2(\mathbb R^n)$,
$$\big\Vert e^{-TF(\vert D_x\vert)}g\big\Vert^2_{L^2(\mathbb R^n)}\le C_{\alpha,T} \int_0^T\big\Vert e^{-tF(\vert D_x\vert)}g\big\Vert^2_{L^2(\omega)}\ \mathrm dt 
+ A_{\alpha} e^{-2\alpha T}\Vert g\Vert^2_{L^2(\mathbb R^n)}.$$
By definition of stabilization, there exist a positive constant $M_{\alpha}\geq1$ and a linear feedback $K_{\alpha}\in\mathcal L\big(L^2(\mathbb R^n)\big)$ such that
\begin{equation}\label{27012021E2}
	\forall t\geq0,\quad \big\Vert e^{-t(F(\vert D_x\vert)+\mathbbm 1_{\omega}K_{\alpha})}\big\Vert_{\mathcal L(L^2(\mathbb R^n))}\le M_{\alpha}e^{-\alpha t}.
\end{equation}
Let $T>0$ fixed. First notice from Duhamel's-type formula for bounded perturbations of semigroups, see e.g. \cite{MR1721989} (Corollary III.1.7), that for all $f\in L^2(\mathbb R^n)$,
$$e^{-T(F(\vert D_x\vert)+\mathbbm 1_{\omega}K_{\alpha})}f = e^{-TF(\vert D_x\vert)}f - \int_0^Te^{-(T-t)F(\vert D_x\vert)} \mathbbm1_{\omega}K_{\alpha}e^{-t (F(|D_x|)+\mathbbm1_{\omega}K_{\alpha})}f\ \mathrm dt.$$
By using the fact that the evolution operators $e^{-tF(\vert D_x\vert)}$ are selfadjoint on $L^2(\mathbb R^n)$, the above formula and \eqref{27012021E2} imply that for all $f,g\in L^2(\mathbb R^n)$,
\begin{align*}
	&\ \big\vert\big\langle e^{-TF(\vert D_x\vert)}g,f\big\rangle_{L^2(\mathbb R^n)}\big\vert = \big\vert\big\langle g, e^{-T F(|D_x|)} f\big\rangle_{L^2(\mathbb R^n)}\big\vert \\
	\le &\ \big\vert\big\langle g, e^{-T (F(|D_x|)+\mathbbm1_{\omega}K_{\alpha})} f\big\rangle_{L^2(\mathbb R^n)}\big\vert + \bigg\vert\int_0^T \big\langle g,  e^{-(T-t) F(|D_x|)} \mathbbm1_{\omega} K_{\alpha} e^{-t (F(|D_x|)+\mathbbm1_{\omega}K_{\alpha})} f \big\rangle_{L^2(\mathbb R^n)}\ \mathrm dt\bigg\vert \\
	\le &\ \big\vert\big\langle g, e^{-T (F(|D_x|)+\mathbbm1_{\omega}K_{\alpha})} f\big\rangle_{L^2(\mathbb R^n)}\big\vert + \int_0^T \big\vert\big\langle\mathbbm1_{\omega}e^{-(T-t) F(|D_x|)}g,K_{\alpha} e^{-t (F(|D_x|)+\mathbbm1_{\omega}K_{\alpha})}f\big\rangle_{L^2(\mathbb R^n)}\big\vert\ \mathrm dt \\
	\le &\ M_{\alpha}e^{-\alpha T} \Vert f\Vert_{L^2(\mathbb R^n)} \Vert g\Vert_{L^2(\mathbb R^n)} + M_{\alpha}\Vert K_{\alpha}\Vert_{\mathcal L(L^2(\mathbb R^n))} \Vert f\Vert_{L^2(\mathbb R^n)} \int_0^T \big\Vert e^{-t F(\vert D_x\vert)} g \big\Vert_{L^2(\omega)}\ \mathrm dt.
\end{align*}
By applying the above inequality to $f = e^{-TF(\vert D_x\vert)}g$ , it follows that for all $g\in L^2(\mathbb R^n)$,
$$ \big\Vert e^{-TF(\vert D_x\vert)}g\big\Vert_{L^2(\mathbb R^n)} \le M_{\alpha}\Vert K_{\alpha}\Vert_{\mathcal L(L^2(\mathbb R^n))} \int_0^T\big\Vert e^{-tF(\vert D_x\vert)}g\big\Vert_{L^2(\omega)}\ \mathrm dt + M_{\alpha}e^{-\alpha T} \Vert g\Vert_{L^2(\mathbb R^n)}.$$
Finally, we deduce from H\"older's inequality and the classical convexity inequality 
$$\forall a,b>0,\quad (a+b)^2\le2(a^2+b^2),$$
that for all $g\in L^2(\mathbb R^n)$,
$$\big\Vert e^{-TF(\vert D_x\vert)}g\big\Vert^2_{L^2(\mathbb R^n)} \le 2M^2_{\alpha}\Vert K_{\alpha}\Vert^2_{\mathcal L(L^2(\mathbb R^n))} T\int_0^T\big\Vert e^{-tF(\vert D_x\vert)}g\big\Vert^2_{L^2(\omega)}\ \mathrm dt + 2M^2_{\alpha}e^{-2\alpha T}\Vert g\Vert^2_{L^2(\mathbb R^n)},$$
and this concludes the proof of Proposition~\ref{18122020T2}.

\end{document}